\documentclass[12pt,a4paper, oneside]{article}
\usepackage[utf8]{inputenc}
\usepackage[T1]{fontenc}
\usepackage{lmodern}
\usepackage[ngerman, english]{babel}
\usepackage{amsmath}
\usepackage{amssymb}
\usepackage{amsthm}
\usepackage{graphicx}
\usepackage{enumitem}
\usepackage{tikz-cd} 
\usepackage{cleveref}
\usepackage[Option]{overpic}
\graphicspath{{figures/}} % Location of the graphics files

\author{ Luca Arcidiacono\thanks{Zentrum Mathematik, Fakult\"at f\"ur Mathematik, Technische Universit\"at M\"unchen,
Boltzmannstr. 3, D-85748 Garching bei M\"{u}nchen}~, Maximilian Engel\footnotemark[1]~, Christian Kuehn\footnotemark[1]}
 
\title{Discretized Fast-Slow Systems \\ 
near Pitchfork Singularities}
\date{\today}

\numberwithin{equation}{section}
\newtheorem{theorem}{Theorem}[section]

\newtheorem{lemma}[theorem]{Lemma}
\newtheorem{proposition}[theorem]{Proposition}

\newtheorem{remark}[theorem]{Remark}

\newcommand{\fa}          {\quad \text{for all } \,}
\newcommand{\rme}{\mathrm{e}}
\newcommand{\rmd}{\mathrm{d}}
\newcommand{\rmD}{\mathrm{D}}
\newcommand{\nwc}{\newcommand}
\nwc{\red}[1]{\textcolor{red}{#1}}
\nwc{\blue}[1]{\textcolor{blue}{#1}}
\nwc{\green}[1]{\textcolor{olive}{#1}}
\nwc{\txtin}{\textnormal{in}}
\nwc{\txtout}{\textnormal{out}}

\let\epsilon=\varepsilon
\newcommand{\R}{\mathbb{R}}
\newcommand{\bigo}[1]{\mathcal{O}\left(#1\right)}
\def\txta{{\textnormal{a}}}

\def\txtr{{\textnormal{r}}}

\DeclareMathOperator*{\argmin}{arg\,min}
\DeclareMathOperator*{\dist}{dist}

\begin{document}

\maketitle

\begin{abstract}
Motivated by the normal form of a fast-slow ordinary differential equation exhibiting a pitchfork singularity we consider the discrete-time dynamical system that is obtained by an application of the explicit Euler method. Tracking trajectories in the vicinity of the singularity we show, how the slow manifold extends beyond the singular point and give an estimate on the contraction rate of a transition mapping. The proof relies on the blow-up method suitably adapted to the discrete setting where a key technical contribution are precise estimates for a cubic map in the central rescaling chart. 
\end{abstract}

{\bf Keywords:} pitchfork bifurcation, slow manifolds, invariant manifolds,  
loss of normal hyperbolicity, blow-up method, discretization.

{\bf Mathematics Subject Classification (2010):} 34E15, 37M99, 37G10, 34C45, 39A99.

\subsection*{Acknowledgments}
This work was supported by the German Research Foundation (DFG) via the SFB-TRR 109. Luca Arcidiacono acknowledges support from the graduate program TopMath of the Elite Network of Bavaria and the TopMath Graduate Center of TUM Graduate School at Technische Universität München.

\section{Introduction} \label{sec:intro}
We study the dynamical system generated by the two-dimensional cubic polynomial map 
\begin{equation}\label{disc}
P: 
\begin{pmatrix}
x \\ y
\end{pmatrix} \mapsto 
\begin{pmatrix}
\bar{x} \\ \bar{y}
\end{pmatrix} = 
\begin{pmatrix}
x+h \left( x(y-x^2)+\lambda\epsilon \right)  \\
y+h\epsilon
\end{pmatrix}
\end{equation}
\\
for small \(\epsilon, h > 0\). The parameter $\epsilon$ implies a time scale separation between the \emph{fast variable} $x$ and the \emph{slow variable} $y$, while the parameter $h$ is viewed as the stepsize in an explicit Euler method for the ordinary differential equation (ODE) 
\begin{align} \label{cont}
\begin{array}{rcrcl}
 \frac{\rmd x}{\rmd \tau} &=& \dot{x} &=& x(y-x^2)+\lambda\epsilon\,, \\
 \frac{\rmd y}{\rmd \tau} &=& \dot{y} &=& \epsilon\,,
\end{array}
\end{align}
which is the normal form of a fast-slow system exhibiting a \emph{pitchfork singularity} at the origin $(x,y) = (0,0)$. Indeed, for $\epsilon = 0$, we can view $y$ as a bifurcation parameter for the flow in the $x$-variable: when $y < 0$, equation~\eqref{cont} has a hyperbolic sink at $(0, y)$ and, when $y > 0$, there are three equilibria with $(0, y)$ unstable and the other two, $(- \sqrt{y},y)$ and $(\sqrt{y},y)$, locally asymptotically stable. The origin $(x,y) = (0,0)$ is called singular since hyperbolicity is lost at this point, and this is also the case for the map~\eqref{disc}. We will analyze the dynamics close to the origin for small $\epsilon, h > 0$. Since we focus on the local behaviour around the singularity, we will neglect potential higher order nonlinearities in the majority of this work, but will show how to adapt the proof when including those. 

For our analysis we will make use of the \emph{blow-up} method~\cite{Du78,Du93}, which has turned out to be a successful tool for treating singular points of fast-slow systems. It was first applied to fast-slow systems by Dumortier and 
Roussarie~\cite{DuRo96} to gain insight in the dynamics around non-hyperbolic equilibria. The method uses a non-injective transformation that maps a higher dimensional object such as a sphere onto the non-hyperbolic equilibrium constituting the singularity. The dynamics on this larger, blown up version of the singularity may then be desingularized by an appropriate rescaling of time and exhibit (partially) hyperbolic behaviour. Then one can use dynamical systems techniques to analyze the dynamics in blown-up space. Finally, a typical result allows one to extend invariant manifolds past the singular point in the blown down system; see e.g.~\cite[Chapter~7]{ku2015} for an introduction and~\cite{DeMaesschalckDumortier7,
DeMaesschalckDumortier4,DeMaesschalckWechselberger,GucwaSzmolyan,ks2011,KuehnUM,KuehnHyp} for an, of course non-exhaustive, list of different applications to 
planar fast-slow systems.

By means of the blow-up method Krupa and Szmolyan \cite{ks2011, ks2001/2} analyze different kinds of singularities in fast-slow ODEs, i.e. fold, canard, transcritical and pitchfork singularities, and show how certain invariant manifolds, so-called \emph{slow manifolds}, extend around the singular points for small $\epsilon > 0$. In the case of fold points, Nipp and Stoffer~\cite{ns2013} transform the blow-up technique to the corresponding explicit Runge-Kutta, in particular Euler, discretization and prove the extension of slow manifolds for the discrete time system around the singularity. Whereas they apply an abstract existence theory for invariant manifolds developed in~\cite{ns2013}, Engel and Kuehn \cite{EngelKuehn18} use direct estimates in the blow-up charts to prove the extension of slow manifolds for transcritical singularities. In both cases, a crucial aspect of the discretized blow-up lies in finding the right rescaling of the step size \(h\). 

In a similar spirit, we investigate, how trajectories of~\eqref{disc} behave near the origin and show how the slow manifold may be continued beyond the pitchfork singularity in the discrete setting. We prove that, depending on the sign of $\lambda$, trajectories starting in the vicinity of the single slow branch near $\{  (0,y) \in \mathbb{R}^2 \,:\, y < 0 \}$ are attracted exponentially to one of the parabolic branches near $\{  (\pm \sqrt{y},y) \in \mathbb{R}^2 \,:\, y > 0 \}$. Furthermore, for $\lambda=0$, we show that canard-type orbits can track the unstable branch $\{  (0,y) \in \mathbb{R}^2 \,:\, y > 0 \}$. Our analysis uses three charts that cover different parts of the blown up space around the singularity. We track trajectories through several checkpoints along a curve of fixed points of the cubic map and give estimates on the contraction of the transition mappings. In this way, we also give an alternative way of proof to the result in \cite{ks2001/2} for the ODE case, by letting the step size $h \to 0$ in~\eqref{disc}.

This paper is structured as follows. After giving a short introduction to continuous time fast-slow systems and pitchfork singularities in Section \ref{sec:backgroundmainresult}, we formulate the setup and main results of this paper at the start of Section~\ref{sec:mainanalysis}. The major part of Section~\ref{sec:mainanalysis} is dedicated to the proof of the main theorem which is divided into several steps. We start with $\lambda\neq 0$. The relevant coordinate changes are discussed in Section~\ref{sec:charts}. Sections~\ref{sec:entering} and~\ref{sec:exiting} describe the dynamics in the vicinity of the branches of the critical manifold, which allows to define the slow manifolds and control contractivity of the transition map. In Section~\ref{sec:rescaling} we describe the continuation of a slow manifold through the blown-up singularity by direct estimates on the trajectories. Finally Section~\ref{sec:blowdown} combines the results in a blown down version, which finishes the proof for $\lambda\neq 0$. The required modifications to cover the canard case $\lambda=0$ are outlined in Section~\ref{canards}. Section \ref{sec:hot} shows how the previous results can be adapted to a more general setting.

\section{Continuous-time fast-slow systems with pitchfork singularity} 
\label{sec:backgroundmainresult}

\subsection{Fast-slow ODEs}
Fast slow systems occur in various fields of science such as neurobiology or chemistry and are usually found as a system of ODEs with two time scales, this means, they are of the form 
\begin{align} \label{fastsys}
\begin{array}{rcrcl}
 \frac{\rmd x}{\rmd t} &=& x' &=& f(x,y,\epsilon)\,, \\
\frac{\rmd y}{\rmd t}&=& y' &=& \epsilon g(x,y,\epsilon)\,, \quad \ x \in \mathbb{R}^m, 
\quad y \in \mathbb R^n, \quad 0 < \epsilon \ll 1\,,
\end{array}
\end{align}
where $f,g,$ are $C^k$-functions with $k \geq 3$. The small parameter $\epsilon$ consitutes the separation between two time scales. 
The variables $x$ and $y$ are often called the \textit{fast} variable(s) and 
the \textit{slow} variable(s) respectively. The time variable 
in~\eqref{fastsys}, denoted by $t$, is termed the \textit{fast} time scale. 
By a change of variables, one can also consider the \emph{slow} time scale \(\tau = \epsilon t\) and rewrite \eqref{fastsys} as
\begin{align}\label{slowsys}
\begin{array}{rcrcl}
\epsilon \frac{\rmd x}{\rmd \tau} &=& \epsilon \dot{x} &=& f(x,y,\epsilon)\,, \\
\frac{\rmd y}{\rmd \tau}&=& \dot{y} &=&  g(x,y,\epsilon)\,.
\end{array}
\end{align} 
The singular limit \(\epsilon = 0\) can be seen from two different perspectives corresponding with the two time scales. 
Setting \(\epsilon = 0 \) in \eqref{fastsys}
yields 
 \begin{align}\label{layer}
\begin{split}
x'  &=  f(x,y,0)\,, \\
y'  &= 0\,,
\end{split}
\end{align}
which is called the \textit{layer problem} (or \textit{fast subsystem}), since we can view the equation layer-wise parametrized by the constant \(y\).
Setting \(\epsilon = 0\) in \eqref{slowsys} gives the differential algebraic equations
\begin{align}\label{reduced}
\begin{split}
0 &=  f(x,y,0)\,, \\
y' &=  g(x,y,0)\,,
\end{split}
\end{align}
called the \textit{reduced problem} (or \textit{slow subsystem}).
The flow of~\eqref{reduced}, the so-called \emph{slow flow}, is restricted to the set 
$$ S_0 := \{(x,y) \in \mathbb{R}^m \times \mathbb{R}^n ~|~~ f(x,y,0) = 0\}\,,$$
which consists of equilibria of the layer 
problem~\eqref{layer}.
We refer to this set as the \emph{critical set} or often also \emph{critical manifold}, in case it is a manifold. 

A subset $S \subset S_0$ is called \emph{normally hyperbolic} if the matrix \(\rmD_xf(x,y,0) \in\mathbb{R}^{m\times m}\) has no eigenvalue with vanishing real part for all \((x,y) \in S\). In the vicinity of normally hyperbolic submanifolds of \(S_0\), the dynamics can be very well described for sufficiently small $\epsilon > 0$: \textit{Fenichel 
Theory} \cite{Fenichel4,Jones,ku2015,WigginsIM} gives the existence of a locally invariant manifold, the \emph{slow manifold} $S_{\epsilon}$, which lies close to \(S_0\) and maintains the stability properties of the layer problem~\eqref{layer}. Furthermore, the restriction of~\eqref{slowsys} to $S_{\epsilon}$ is a regular perturbation of the reduced problem~\eqref{reduced}.

However, points $p \in \mathbb{R}^m \times \mathbb{R}^n$, which do not satisfy normal hyperbolicity are called \emph{singularities} in this context and are more delicate to handle. From the view point of the layer equation~\eqref{layer} singularities often correspond to bifurcations of the fast subsystem, and the breakdown of normal hyperbolicity is typically associated with the intersection of multiple parts of \(S_0\) at the point $p$ where degeneracy of \(\rmD_xf(p)\) follows from the Implicit Function Theorem. In the case of a pitchfork singularity, which we will consider in the following, we are precisely in such a situation.

\subsection{Pitchfork singularity in continuous time}

We consider a two-dimensional fast-slow system of the form \eqref{fastsys} where the critical manifold resembles the shape of a pitchfork, and we call the associated non-hyperbolic singularity a \emph{pitchfork singularity}.
Such a situation occurs when the vector field $f$ satisfies
\[ \begin{array}{rr}
		f(0,0,0)=0, & \qquad	\partial_xf(0,0,0)=0,\\
	\partial_yf(0,0,0)=0, &\qquad 	\partial_{xx}f = 0, \\
	\partial_{xxx}f(0,0,0) \neq 0,  &\qquad  \partial_{xy}f(0,0,0)\neq 0.
\end{array} \]
These conditions guarantee that, for \(\epsilon = 0\), there is a non-hyperbolic equilibrium at the origin where the critical manifold has a transversal self-intersection and one part of the branches crosses the other tangentially to the $x$-direction. In particular, we assume that \[ \partial_{xxx}f(0,0,0) < 0, \qquad  \partial_{xy}f(0,0,0) > 0 ~,\] such that the singularity is \emph{supercritical}.
Furthermore, we assume \(g(0,0,0) > 0\) such that the slow dynamics pass through the origin in positive $y$-direction. In other words, we consider the problem of how the slow dynamics behave in the vicinity of a splitting into three critical branches (see \cite[Figure 4]{ks2001/2}). The case of a \emph{subcritical} pitchfork singularity or the situation of \(g(0,0,0) < 0 \) are less challenging, since the dynamics only heads into the direction of one critical branch, and will therefore not be treated in this paper. 

There is a linear change of coordinates (see \cite{ks2001/2}) which brings the system into the normal form
\begin{align}\label{eq:normalform}
\begin{split}
\dot{x} &= x(y-x^2)+\lambda\epsilon+ h_1(x,y,\epsilon)\,, \\
\dot{y} &= \epsilon(1+h_2(x,y,\epsilon))\,,
\end{split}	
\end{align}
where \(h_1 \) and \(h_2\) satisfy \(h_1(x,y,\epsilon) = \bigo{x^2y, xy^2, \epsilon x, \epsilon y, \epsilon^2}, ~ h_2(x,y,\epsilon) = \bigo{x,y,\epsilon}\).
Since we are mainly interested in the local dynamics around the origin, we may initially ignore the higher order terms and only consider the system
\begin{align}\label{continuous2}
\begin{split}
\dot{x} &= x(y-x^2)+\lambda\epsilon\,, \\
\dot{y} &= \epsilon\,.
\end{split}	
\end{align} 
The critical manifold is given as \[ S_0 = \{(x,y) \in \mathbb{R}^2 \,:\, y=x^2\} \cup \{(x,y) \in \mathbb{R}^2 \,:\, x = 0\} ~.\]
For negative \(y\) there is only one stable equilibrium at $x = 0$, while for positive \(y\) we have an unstable equilibrium at \(x=0\) and two locally asymptotically stable ones at \(x= \pm \sqrt{y}\).

\section{Pitchfork singularity in discrete time} \label{sec:mainanalysis}
A time-discretization of equation~\eqref{continuous2} by the explicit Euler method with time step size \(h > 0\) yields the map
\begin{equation}\label{disc2}
P: 
\begin{pmatrix}
x \\ y
\end{pmatrix} \mapsto 
\begin{pmatrix}
\bar{x} \\ \bar{y}
\end{pmatrix} = 
\begin{pmatrix}
x+h \left( x(y-x^2)+\lambda\epsilon \right)  \\
y+h\epsilon
\end{pmatrix}\,.
\end{equation}
As in continuous time, the system induced by~\eqref{disc2} clearly possesses the critical manifold
$$ S_0 = \{(x,y) \in \mathbb{R}^2 \,:\, y=x^2\} \cup \{(x,y) \in \mathbb{R}^2 \,:\, x = 0\}\,,$$
consisting of fixed points of~\eqref{disc2} for $\epsilon = 0$. We split the set $S_0$ into the four branches
\begin{eqnarray}
&&S_{a}^0 = \{x=0, y < 0\}, \qquad S_{a}^+ = \{y=x^2, x > 0\}, \nonumber\\
&&S_{a}^- = \{y=x^2, x < 0\}, \qquad S_{r}^0 = \{x=0, y > 0\}. \label{eq:Snot}
\end{eqnarray}
By linearization we see that these four branches are normally hyperbolic as long as for $(x,y)\in S_0\setminus\{(0,0)\}$ we have 
\begin{equation}
\label{eq:condbigcutoff}
|1+hy|\neq 1.
\end{equation}
Since we want to restrict the analysis to the non-hyperbolic singularity at the origin $(x,y)=(0,0)$, we always assume that $h$ is chosen small enough so that~\eqref{eq:condbigcutoff} holds as well as the same stability properties as in the time-continuous case. For example, for a fixed initial condition with $y_0<0$, we have to ensure $1+hy_0>-1$ which yields the restriction $h<2/|y_0|$, which then implies that $S_{a}^0$ is normally hyperbolic and locally attracting. Note that we shall still use the notation in~\eqref{eq:Snot} in this context.  In contrast to the continuous case the fixed points on \(S_a^0\) are not globally stable, but only inside an interval around zero of size \(\sqrt{\frac{2}{h}+y}\). Outside this interval solutions diverge in growing oscillations. Compare also with the reasoning of Lemma \ref{lem:transition1} in Section \ref{sec:entering}.
 Due to normal hyperbolicity and according to~\cite[Theorem 4.1]{HPS77}, for \(\epsilon \) sufficiently small, there exist corresponding invariant slow manifolds  
\(S_{a, \epsilon, h}^0 \)
\(S_{a, \epsilon, h}^+ \)
\(S_{a, \epsilon, h}^- \)
\(S_{r, \epsilon, h}^0 \).
However, the four branches of the critical manifold $S_0$ intersect at the origin $(x,y)=(0,0)$, where we have
\(\rmD_x P(0, 0) = 1 \), i.e.~we observe the loss of normal hyperbolicity as in the ODE case.
\subsection{Main result}
We want to investigate where points around \(S_{a, \epsilon, h}^0 \) get mapped to by iterations of \(P\) in order to find the continuation of \(S_{a, \epsilon, h}^0 \) beyond the singularity. For that purpose, fix some \(\rho  > 0 \), let \(J \subset \R\) be a small interval containing $0$ and define the section \(\Delta_{\textnormal{in}}\) around the point \((0,-\rho^2)\) on the critical branch \(S_a^0\) by
 \[ \Delta_{\textnormal{in}} = \{ (x,y) \in \mathbb{R}^2 ~| ~ y = - \rho^2, x \in J \} ~.\] 
In particular, we always assume that the initial condition is chosen on $\Delta_{\textnormal{in}}$ and $J$ is sufficiently small so that trajectories effectively start 
close to the attracting slow manifold \(S_{a, \epsilon, h}^0 \).

We are going to follow trajectories of \eqref{disc} starting in \(\Delta_{\textnormal{in}}\) up to height \(y= \rho^2\). Since the line \(\{y=\rho^2\}\) 
can only be reached in case \(\frac{2\rho^2}{\epsilon h} \in \mathbb{N}\), we introduce \(\tilde{\rho} \) as the closest reachable height, which then  satisfies \( |\rho ^2 - \tilde{\rho}^2 | < \epsilon h\). This allows us to define the sections
\begin{equation*} \label{Deltaout}
\Delta_{\textnormal{out}}^0 = \{ (x,y) \in \mathbb{R}^2 ~| ~ y = \tilde{\rho}^2, x \in  J \}  ~, ~~\Delta_{\textnormal{out}}^\pm = \{ (x,y) \in \mathbb{R}^2 ~| ~ y = \tilde{\rho}^2, x \mp \tilde{\rho} \in   J \}
\end{equation*} 
around the points \((\tilde{\rho}, \tilde{\rho}^2), (-\tilde{\rho}, \tilde{\rho}^2)\) or \((0, \tilde{\rho}^2)\) on the branches \(S_a^+, S_a^-\) or \(S_r^0\) respectively.
Depending on the sign of \(\lambda\) we will show that the transition maps \(\Pi^\pm: ~  \Delta_{\textnormal{in}} \to \Delta_{\textnormal{out}}^\pm\) or \(\Pi^0 : \Delta_{\textnormal{in}} \to \Delta_{\textnormal{out}}^0\) if \(\lambda =0 \) are well defined. The transition maps are given by the \(N\)-fold of \(P\) where \(N = [\frac{2\rho^2}{\epsilon h}] \) is the closest integer to \(\frac{2\rho^2}{\epsilon h}\).  For the discrete setting, induced by the map~\eqref{disc2}, we have the following main result (see Figure \ref{fig:main} for an illustration of the case \(\lambda >0\)).

\begin{theorem}\label{maintheorem}
Fix \(\rho > 0 \) and let \(\lambda \in \mathbb{R} \) . Then there are \(\epsilon_0, h_0 >0\), depending on $\lambda$, such that for all \(\epsilon \in (0,  \epsilon_0] \) and all \(h \in (0, h_0] \) %[\red{Relation between $\epsilon_0$ and $h_0$?}]
the following holds. 
\begin{enumerate}
\item[(T1)] If \(\lambda >0 \), the set \(\Delta_{\txtin} \) (including the point  \(\Delta_{\txtin} \cap S_{a,\epsilon,h}^0\)) is mapped by \(\Pi^+\) into a subset of \(\Delta_{\txtout}^+\) that contains the point \(\Delta_{\txtout}^+ \cap S_{a,\epsilon,h}^+\) and  has a  width of order \(\bigo{\left(1-c \cdot h\right)^{\frac{K}{\epsilon h}} }\) for some constants \(c, K >0 \). ~~ %\red{[TODO: add dependencies and convergence of the constants ]}
\item[(T2)] If \(\lambda <0 \), an analogous statement holds with \(\Delta_{\txtout}^-\) , \(S_{a,\epsilon,h}^-\) and \(\Pi^-\) instead of \(\Delta_{\txtout}^+\), \(S_{a,\epsilon,h}^+\) and \(\Pi^+\).
\item[(T3)] If \(\lambda = 0\) the slow manifolds \(S_{a,\epsilon,h}^0\) and \(S_{r,\epsilon,h}^0\) coincide with the critical branches \(S_a^0\) and \(S_r^0\) and are connected by a canard solution, i.e. \(\Pi^0(\Delta_{\txtin} \cap S_{a,\epsilon,h}^0 ) \in S_{r,\epsilon,h}^0 \). The set \(\Delta_{\txtin}\) gets mapped by \(\Pi^0\) into \(\Delta_{\txtout}^0\) and its image has a width less than \(\Big(1- h^2 (\tfrac{\rho^2}{2})^2\Big)^{\tfrac{\rho^2}{2\epsilon h}} \cdot |J| \).
\end{enumerate}
\end{theorem}

\begin{remark}
An analogous result for the continuous-time system \eqref{continuous2} has been shown in \cite{ks2001/2}.
The sections are chosen to be \(  \Delta_{\textnormal{out}}^\pm = \{ (x,y) \in \mathbb{R}^2 ~| ~ x = \rho^2,  y \mp \rho^4 \in   J \} \) such that the transition maps \(\Pi^\pm: ~  \Delta_{\txtin} \to \Delta_{\txtout}^\pm\) are well defined if \(\textnormal{sign}(\lambda) =\pm 1\). The image  \(\Pi^\pm(  \Delta_{\txtin}) \subset \Delta_{\txtout}^\pm\) is an interval about the corresponding point on the slow manifold $S_{a, \epsilon}^\pm$ and has a size of  \(\bigo{\rme^{-\frac{C}{\epsilon}}}\) for some constant $C >0$.
\end{remark}

Note that we will assume $h_0 \epsilon_0 \ll \rho$ to obtain meaningful time lengths for the dynamical analysis. For further details on the choice of \(h_0\) see Section \ref{sec:blowdown}, but notice that it immediately implies the stability restriction $h < \frac{2}{\rho^2}$ from the discussion below~\eqref{eq:condbigcutoff}.

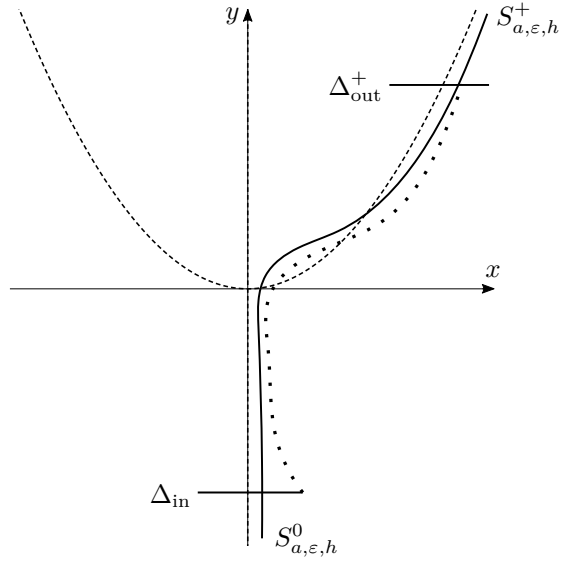
\begin{figure}[ht] 
	\begin{center}

		\begin{overpic}[width=0.4\textwidth]{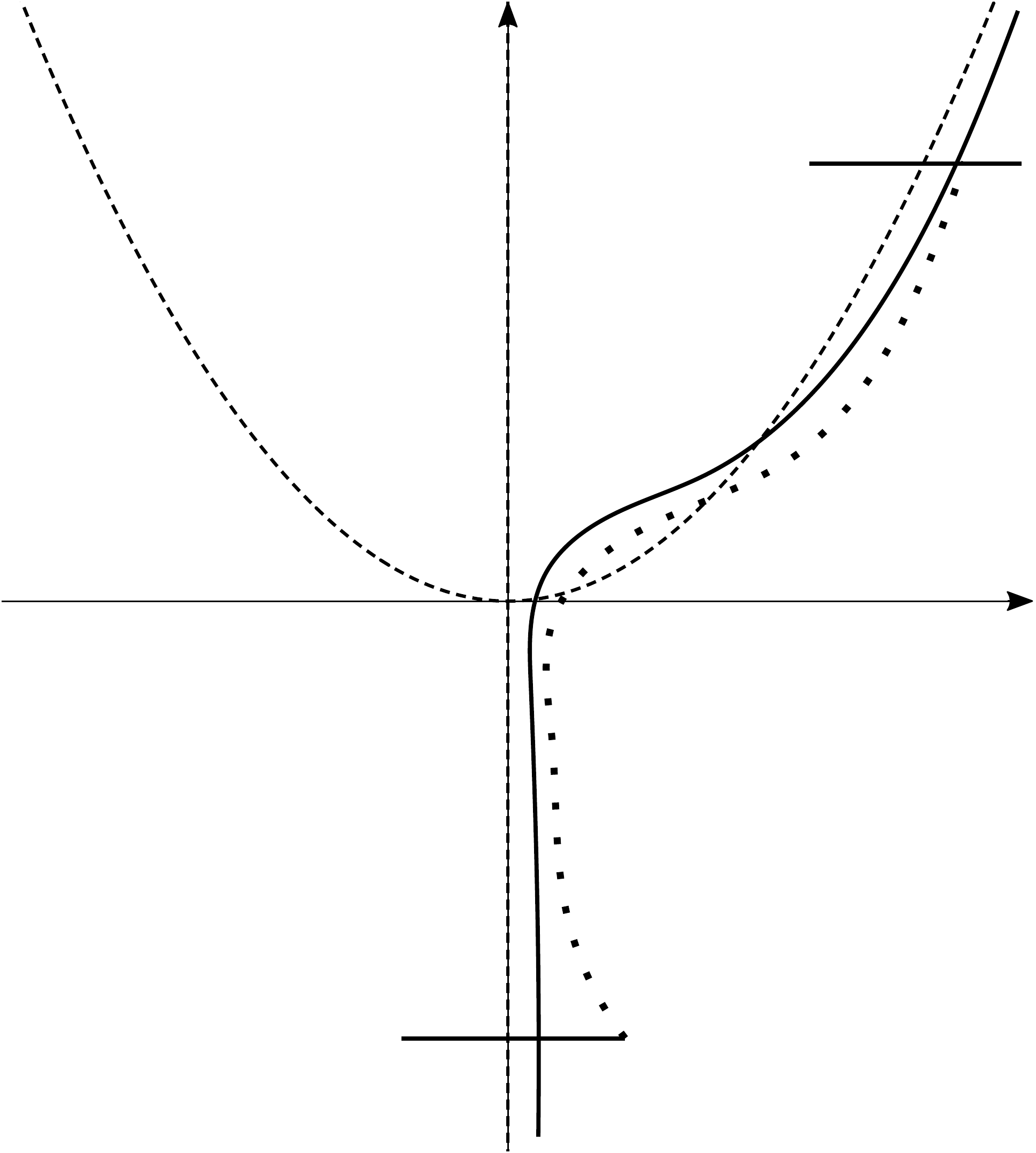}
			\footnotesize{
				\put(88,50){\(x\)}
				\put(40,98){\(y\)}
				\put(26,8){\(\Delta_{\textnormal{in}}\)}
				\put(59,84){\(\Delta_{\textnormal{out}}^+\)}
				\put(49,0){\(S_{a,\epsilon, h}^0\)}
				\put(90,97){\(S_{a,\epsilon, h}^+\)	}
			}		
		\end{overpic}
		
		\caption{Extension of the slow manifold \(S_{a,\epsilon,h}^0\) (bold line)  through the neighbourhood of the origin for \(\lambda >0\). Further the critical manifold \(S_0\) (dashed line) and a sample trajectory  from \(\Delta_{\textnormal{in}} \) to \(\Delta_{\textnormal{out}^+}\) (dotted line). }
		\label{fig:main}
	\end{center}
\end{figure}

\subsection{Transformation to the chart coordinates} 
\label{sec:charts}

The proof of Theorem~\ref{maintheorem} uses the blow-up method for the dynamical system induced by the map

\begin{equation}\label{discrete}
P: ~~\R^4 \to \R^4 ~~~~~~
\begin{pmatrix}
x \\ y\\ \epsilon \\h
\end{pmatrix} \mapsto 
\begin{pmatrix}
\bar{x} \\ \bar{y}\\ \bar{\epsilon} \\ \bar{h}
\end{pmatrix} = 
\begin{pmatrix}
x+h \left( x(y-x^2)+\lambda\epsilon \right)  \\
y+h\epsilon\\
\epsilon\\ h
\end{pmatrix},
\end{equation}
where the fast-slow separation parameter \(\epsilon\) and the stepsize \(h\) are also seen as variables.

The quasi-homogeneous blow-up transformation around the pitchfork singularity is given as
\begin{align*}
&\Phi : B:= \mathbb{S}^2 \times (0, \infty) \times (0, \infty) \to \R^4, ~~~~ (\tilde{x}, \tilde{y}, \tilde{\epsilon}, \tilde{h}, r) \mapsto (x,y,\epsilon, h), \\
&x= r \tilde{x}, \qquad y=r^2\tilde{y}, \qquad \epsilon = r^4\tilde{\epsilon}, \qquad h=r^{-2}\tilde{h}.
\end{align*} 
The transformation of the $(x, y,\epsilon)$-coordinates is the same as in the continuous-time case (see \cite{ks2001/2}). The change of variables in $h$ is chosen such that the map is desingularized in the relevant charts. We exclude $0$ from the domain of $\tilde h$ since at $\tilde h = 0$ every point is a neutral fixed point. Due to the transformation $h = \tilde h/ r$ we have to exclude $0$ from the domain of $ r$ as well. 

The transformation $\Phi$ induces a map \(\bar{P} := \Phi^{-1} \circ P \circ \Phi\) on the manifold $B$. We analyse the dynamics of $\bar{P}$  by using the 
charts $K_i$, $i=1,2,3$, 
\begin{align*}
	&K_1 : D_1 := 
	 \R \times \R^+ \times \R^+_0 \times \R^+	\to~  \R \times \R^- \times \R^+_0 \times \R^+ 
	  ~ & (x_1,r_1, \epsilon_1, h_1) \mapsto (x,y,\epsilon, h)~, \\
 	&K_2 : D_2 :=  
 	\R \times \R \times \R^+ \times \R^+	~~\to ~  \R \times \R \times \R^+ \times \R^+  
 	 ~ &(x_2,y_2, r_2, h_2) \mapsto (x,y,\epsilon, h)~, \\
	&K_3 : D_3 :=  \R \times \R^+ \times \R^+_0\times \R^+ \to~ \R \times \R^+ \times \R^+_0 \times \R^+   ~ & (x_3,r_3, \epsilon_3, h_3) \mapsto (x,y,\epsilon, h)~, 
\end{align*}
which are given by 
\begin{align*}
K_1: \quad x&=r_1x_1 &\quad y&= -r_1^2 &\quad \epsilon &= r_1^4\epsilon_1 &\quad h &= r_1^{-2}h_1~,\\
K_2: \quad x&=r_2x_2 &\quad y&=  r_2^2y_2 &\quad \epsilon &= r_2^4 &\quad h &= r_2^{-2}h_2~,\\
K_3: \quad x&=r_3x_3 &\quad y&= r_3^2 &\quad \epsilon &= r_3^4\epsilon_3 &\quad h &= r_3^{-2}h_3~.
\end{align*}

To switch between different chart coordinates we use the following coordinate changes
\begin{align*}
	\kappa_{12}: \quad x_2&=x_1\epsilon_1^{-\frac{1}{4}} &\quad y_2&= -\epsilon_1^{-\frac{1}{2}} &\quad r_2 &= r_1 \epsilon_1^{\frac{1}{4}} &\quad h_2 &= h_1\epsilon_1^{\frac{1}{2}}~, \\ 
\kappa_{21}: \quad x_1&=x_2(-y_2)^{-\frac{1}{2}} &\quad r_1&= r_2(-y_2)^{\frac{1}{2}} &\quad \epsilon_1 &= y_2^{-2} &\quad h_1 &= -h_2y_2~, 	
\end{align*}
and 
\begin{align*}
\kappa_{23}: \quad x_3&=x_2(y_2)^{-\frac{1}{2}} &\quad r_3&= r_2(y_2)^{\frac{1}{2}} &\quad \epsilon_3 &= y_2^{-2} &\quad h_3 &= h_2y_2~,\\
\kappa_{32}: \quad x_2&=x_3\epsilon_3^{-\frac{1}{4}} &\quad y_2&= \epsilon_3^{-\frac{1}{2}} &\quad r_2 &= r_3 \epsilon_3^{\frac{1}{4}} &\quad h_2 &= h_3\epsilon_3^{\frac{1}{2}}~. 
\end{align*}
For the proof of Theorem \ref{maintheorem} we will proceed as follows. 
Transforming \eqref{discrete} using the coordinate changes \(K_i ~(i=1,2,3)\)~induces a dynamical system on \(D_i \subset \R^4\) in chart coordinates. Trajectories of \eqref{discrete} are analyzed via their corresponding transformed versions in each of the charts. In every chart \(K_i\) we will define sets \(\Sigma_i^{\txtin}\) and \(\Sigma_i^{\txtout}\), show that the transition maps \(\Pi_i : \Sigma_i^{\txtin} \to \Sigma_i^{\txtout}\) are well defined and study their contractivity. The mappings \(\Pi^\pm\) are then built by connecting the three chart-wise transition maps, which are elaborated in Section~\ref{sec:blowdown} in more detail. 
 
We refer to \(K_1\) as the \emph{entering chart}, as we start our analysis in this chart and trajectories are brought closer to the origin. 
Charts of the type of \(K_2\) are often called \emph{rescaling charts}, since the transformation is basically a rescaling with suitable powers of the fast-slow-separation constant \(\epsilon\). In this chart, the dynamics arbitrarily close to the origin are analyzed. Finally the \emph{exiting chart} \(K_3\) is used to describe, how trajectories exit the vicinity of the origin and is crucial for the contractivity statement of Theorem~\ref{maintheorem}.

\subsection{Dynamics in the entering chart}
\label{sec:entering}

Fix some \(\rho > 0\) and also consider the case
\begin{equation*}
\lambda >0.
\end{equation*}
from now on until Section~\ref{canards}. The case \(\lambda <0 \) can be treated analogously, see also Section \ref{sec:blowdown} for more details. Further take \(\epsilon , h > 0 \) sufficiently small. During the next sections we will specify what sufficiently small means for \(\epsilon \) and \( h\) such that \(\epsilon_0 \) and \( h_0 \) are determined. In the coordinates $(x_1, r_1, \epsilon_1, h_1)$ of the first chart $K_1$, the set \(\Delta_{\textnormal{in}}\) is given as
$$ \Sigma_1^{\textnormal{in}} = \left\{ \left(\frac{x}{\rho}, \rho, \delta, \nu \right) \,:\, x \in J\right\}\,, $$
for which we define \[ \delta := \frac{\epsilon}{\rho^4}  ~~\text{ and }~~ \nu :=  h\rho^2 ~.\]
We investigate the dynamics within the domain 
$$D_1 := \left\{(x_1, r_1, \epsilon_1, h_1) \in \mathbb{R}^4 : r_1 \in 
[ 0, \rho], \epsilon_1 \in [0, 16 \delta], h_1 \in [ 0, \nu] \right\}\,.$$

In order to find an expression for the map~\eqref{discrete} in terms of the entering chart \(K_1\), we first rewrite \( \bar{y} =  y + \epsilon h  \) in \(K_1\)-coordinates as
\[ -\bar{r}_1^2 = -r_1^2 + r_1^{-2} h_1 r_1^4 \epsilon_1 ~.\] 
This yields 
\begin{equation}\label{r1relation}
\bar{r}_1^2 = r_1^2 (1 - h_1 \epsilon_1)\,.
\end{equation}

The remaining three equations of~\eqref{discrete} in \(K_1\)-coordinates read as

\begin{align*}
\bar{r}_1\bar{x}_1 &=r_1 x_1 + r_1^{-2} h_1  \left(r_1 x_1 (-r_1^2-r_1^2 x_1^2)+ \lambda r_1^4\epsilon_1 \right) \,,  \\
\bar{r}_1^4\bar{\epsilon}_1 &= r_1^4 \epsilon_1 \,,\\
\bar{r}_1^{-2}\bar{h}_1 &=r_1^{-2} h_1 \,.
\end{align*}
Hence, by using \eqref{r1relation}, we obtain the maps
\begin{align} \label{K1_dynamics}
\begin{array}{rcrcl}
&\bar{x}_1& &=&(1-h_1\epsilon_1)^{-\frac{1}{2}} \left[ x_1+ h_1\left(x_1(-1-x_1^2)+\lambda r_1\epsilon_1 \right) \right] ~,\\
&\bar{r}_1& &=& (1-h_1\epsilon_1)^{\frac{1}{2}}~ r_1~,\\
&\bar{\epsilon}_1& &=& (1-h_1\epsilon_1)^{-2} \epsilon_1~,\\
&\bar{h}_1& &=&(1-h_1\epsilon_1)~ h_1~.
\end{array}
\end{align}

The dynamics in \(r_1, \epsilon_1\) and \(h_1\) can be calculated explicitly for the first chart. 
\begin{lemma}\label{r1h1eps1}
For $\xi_0 := \frac{1}{h_1(0)\epsilon_1(0)} > 0$, the trajectories of~\eqref{K1_dynamics} in $r_1, \epsilon_1, h_1$ are given by
\begin{equation*} 
	r_1(n) = r_1(0) \sqrt{ \frac{\xi_0-n}{\xi_0}} ~, \ 
	\epsilon_1(n) = \epsilon_1(0) \frac{\xi_0^2} {(\xi_0-n)^2} ~, \ h_1(n) = h_1(0) \frac{\xi_0-n}{\xi_0}\,,
\end{equation*}
for $n\in \mathbb{N} , n < \xi_0 .$
\end{lemma}
\begin{proof}
We define \(\eta := h_1 \epsilon_1\). Multiplying the last two equations of \eqref{K1_dynamics} gives the relation 
\[\bar{\eta} = \bar{h}_1 \bar{\epsilon}_1  = (1-h_1\epsilon_1)^{-1} h_1 \epsilon_1 = (1-\eta)^{-1}\eta .\]
Given the initial condition \(\eta(0) = \xi_0^{-1} > 0 \),  we obtain the solution \(\eta(n)= (\xi_0-n)^{-1}\) .
Hence, we can compute
\[h_1(n) = h_1(0) \prod_{k=0}^{n-1} (1-\eta(k)) =h_1(0) \prod_{k=0}^{n-1} \frac{\eta(k)}{\eta(k+1)} = h_1(0) \frac{\eta(0)}{\eta(n)} =h_1(0) \frac{\xi_0-n}{\xi_0}\,.\]
Similarly, we obtain the formulas for $r_1$ and $\epsilon_1$.
\end{proof}
Furthermore, we can observe from equations~\eqref{K1_dynamics} that the set \(\{\epsilon_1=0\}\) is invariant for the dynamics and, for given $r_1^*, h_1^*$, consists of the two-parameter family of invariant one-dimensional lines \[\{\epsilon_1=0, r_1=r_1^*,  h_1=h_1^*\}.\]
Each of these lines has a fixed point located at \((x_1, r_1, \epsilon_1, h_1)= (0, r_1^*, 0, h_1^*)\),
which has a three-dimensional center eigenspace and a one-dimensional stable eigenspace in \(x_1\)-direction with eigenvalue \(1-h_1^*\) (recall that $h_1^* \leq h \rho^2 < \frac{1}{3} <2$). In other words, 
the two-dimensional plane
$$ S_{\txta,1}^{0} = \{(x_1, r_1, \epsilon_1, h_1) \in D_1 \,:\, x_1 = 0,\epsilon_1 = 0 \}$$
is an invariant manifold in $D_1$ only consisting of fixed points, attracting in the 
$x_1$-direction and neutral in the other directions, and corresponding to the branch \(S_{\txta}^0\) of the critical manifold. 
In particular, for each $h_1 \geq 0$ 
we have the fixed point
$$ p_a^0(h_1) = (0,0,0, h_1)\,.$$ 
We obtain the following statement:

\begin{proposition} \label{centermanifold_K1}
The invariant manifold $S_{\txta,1}^{0}$ extends to a center-stable invariant manifold \(M_{\txta,1}^0\) (at $p_a^0(0)$) which is given in $D_1$ by a graph $x_1 = l_1(r_1, \epsilon_1, h_1)$ for some mapping $l_1$. Furthermore, $\epsilon_1$ is increasing in $D_1$ (and thereby, in particular, on $M_{\txta,1}^{0}$), 
whereas $h_1, r_1$ are decreasing in $D_1$ (and thereby also on $M_{\txta,1}^{0}$).
\end{proposition}

\begin{proof}
This is an immediate consequence of the considerations above and classical center manifold theory (cf.~e.g.~\cite[Chapter 5A]{HPS77}). From Lemma~\ref{r1h1eps1}, we can see immediately that as long as $r_1(n), h_1(n) > 0$, we have $ n < \xi_0$.  Hence, the claim follows from the formulas in Lemma~\ref{r1h1eps1}.
\end{proof} 

Note that, on $\{r_1 > 0, \epsilon_1 > 0, h_1 > 0 \}$, the manifold \(M_{a,1}^0\) corresponds to the union of the slow manifolds \(S_{a, \epsilon, h}^0\). 
Assume we iterate system~\eqref{K1_dynamics} until \(r_1(n)\) reaches a value less or equal to \(\frac{\rho}{2}\). (This means that the trajectory of~\eqref{discrete} has reached a point above the line \(y=-\rho^2/4\) in \(x,y\)-coordinates.)
As already remarked before, when defining \(\Delta_{\textnormal{out}}^\pm\), the level \(\frac{\rho}{2}\) might not be hit precisely. However, let us assume for simplicity that we are in the situation of reaching \(r_1 = \frac{\rho}{2}\) after $n^* \in \mathbb{N}$ iterates. (Also in the next sections we will assume in a similar way that specific values are hit since the small errors do clearly not change our results.) This means we have \(\frac{\xi_0-n^*}{\xi_0} =  \frac{1}{4}\), and reversely, \(n^* = \frac{3}{4} \xi_0 = \frac{3}{4\nu \delta}\). 

Denote by \(\Pi_1: \Sigma_1^{\txtin} \to \R^4\) the transition map after \(\frac{3}{4\nu \delta}\) iterations. 
We can deduce the following Lemma:
\begin{lemma} \label{lem:transition1}
For $\delta > 0$ sufficiently small, we have
$$\Pi_1(\Sigma_1^{\textnormal{in}}) \subset \Sigma_1^{\textnormal{out}} := \left\{x_1 \in [-1,1],  r_1= \frac{\rho}{2}, \epsilon_1 = 16 \delta, h_1 = \frac{\nu}{4} \right\}\,, $$
and \(M_{\txta, 1}^0 \cap \Sigma_1^{\textnormal{in}}\) as well as \(M_{\txta, 1}^0 \cap \Sigma_1^{\textnormal{out}}\) are non-empty sets.
\end{lemma}
\begin{proof}
From the explicit solutions in Lemma~\ref{r1h1eps1}, we obtain  \[\Pi_1(\Sigma_1^{\txtin}) \subset \left\{r_1= \frac{\rho}{2}, \epsilon_1 = 16 \delta, h_1 = \frac{\nu}{4} \right\} .\]

For \(\epsilon = 0\) the \(x_1\)-equation of \eqref{K1_dynamics} reads \begin{equation}\label{eq:x1_for_eps=0}
	\bar{x}_1 = x_1+ h_1\left(x_1(-1-x_1^2)\right) = (1-h)x_1 - hx_1^3 ~.
\end{equation} 
Consequently we have \begin{equation*}
	\bar{x}_1 < (1-h_1)x_1 ~~\text{ for } x_1 > 0  ~~\text{ and }~~ \bar{x}_1 > (1-h_1)x_1 ~~\text{ for } x_1 < 0 ~.
\end{equation*}
Additionally a direct computation yields that for \(|x_1| \leq \sqrt{\frac{2}{h}-2} \) we have \[ |\bar{x}_1| \leq |(1-h_1)x_1| \]
This means that \(\bar{x}_1 \) lies in the cone bordered by \(\pm(1-h_1)x_1\) for all \(x_1\) that satisfy \(|x_1| \leq \sqrt{\frac{2}{h_1}-2} \) (see also Figure \ref{fig:convergence:cone}). Thus contraction towards the fixed point at 0 is guaranteed for all initial values of \(x_1\) in that domain and the speed is at least the linear rate \(1-h_1\).

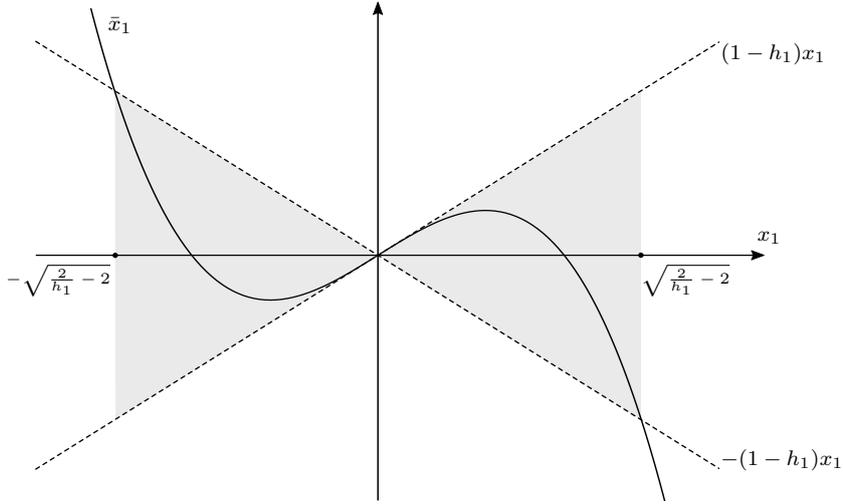
\begin{figure}[ht] 
	
	\begin{center}
		
		\begin{overpic}[width=0.6\textwidth]{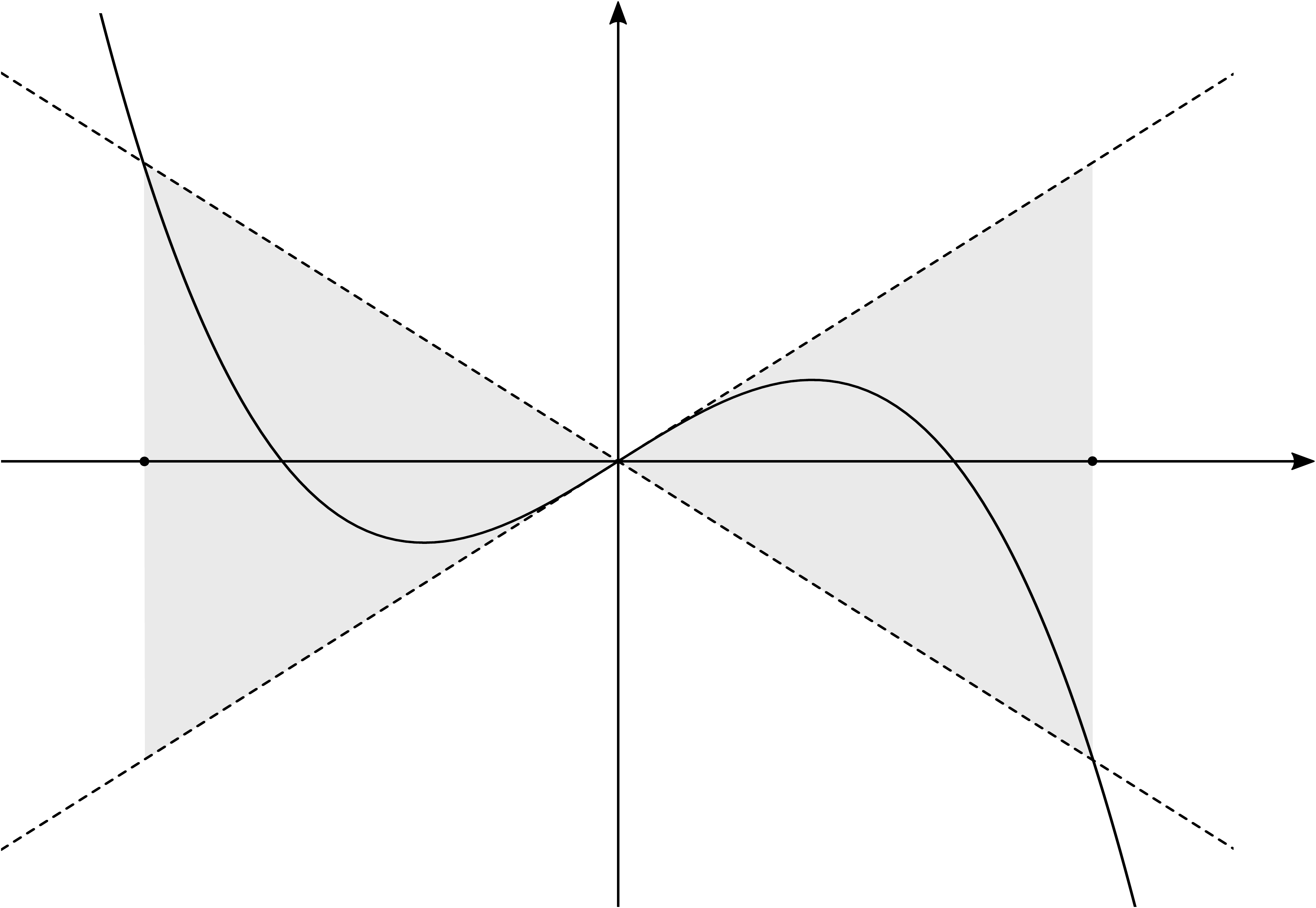}
			\scriptsize
			\put(99,36){\(x_1\)}
			\put(10,65){\(\bar{x}_1\)}
			\put(94,61){\((1-h_1)x_1\)}
			\put(94,5){\(-(1-h_1)x_1\)}
			\tiny \put(-4,30){\(-\sqrt{\tfrac{2}{h_1}-2}\)}
			\tiny \put(83,30){\(\sqrt{\frac{2}{h_1}-2}\)}
		\end{overpic}
	\caption{	For \(x_1 \in \left[-\sqrt{\frac{2}{h_1}-2}, \sqrt{\frac{2}{h_1}-2}~ \right] \) the image \(\bar{x}_1 = x_1+ h_1\left(x_1(-1-x_1^2)\right)\) lies inside the cone spanned by \(\pm (1-h_1)x_1\). } \label{fig:convergence:cone}
	\end{center}
\end{figure}

Note that \(h_1(n)\) stays inside the interval \(\left[\frac{\nu}{4}, \nu \right]\) for $n \leq n^*$ and the stable eigenvalue in \(x_1\)-direction is given by \(1-h_1\) along the manifold $S_{\txta,1}^{0}$.  We may choose \(\nu\) small enough so that \( \rho^{-1}J \subset  \left[-\sqrt{\frac{2}{\nu}-2}, \sqrt{\frac{2}{\nu}-2}~\right] \).
Hence it follows from standard perturbation arguments that, for \(\delta\) sufficiently small, the map \(\Pi_1\) is a contraction with rate \( (1-c) ^ {\frac{3}{4\nu \delta}}\) for some constant \(c < \frac{\nu}{4}\).  Therefore we have \(\Pi_1(\Sigma_1^{\textnormal{in}}) \subset \Sigma_1^{\textnormal{out}}\) for a sufficiently small choice of \(\delta\). 

Since for \(\delta \to 0 \) the point \(M_{\txta, 1}^0  \cap \{\epsilon_1=\delta\}\) approaches \((0,\rho, 0, \nu)\), we can deduce that \(M_{\txta, 1}^0 \cap \{\epsilon_1=\delta\}\) lies in \(\Sigma_1^{\txtin}\) for \(\delta\) small enough. 
Additionally, since the manifold \(M_{\txta, 1}^0 \) is invariant under the forward iterations of~\eqref{K1_dynamics}, we have that \(M_{\txta, 1}^0  \cap \Sigma_1^{\txtout}\) is non-empty as well. 
\end{proof}

\subsection{Dynamics in the rescaling chart}
\label{sec:rescaling}

We use \(\kappa_{12}\) to transfer the set \(\Sigma_1^{\txtout}\) to the second chart and define 

\[ \Sigma_2^{\txtin} := \kappa_{12}(\Sigma_1^{out } ) = \left\{x_2 \in \left[ -\frac{1}{2}\delta^{-\frac{1}{4}}, \frac{1}{2}\delta^{-\frac{1}{4}} \right],  y_2 = -\frac{1}{4}\delta ^{-\frac{1}{2}} , ~~ r_2 = \epsilon^{\frac{1}{4}}= \rho \delta^{\frac{1}{4}}, ~~ h_2 = \nu \delta^{\frac{1}{2}} \right\}. \]

For the transformation of \eqref{discrete} via the chart \(K_2\), first observe that since \(\bar{\epsilon} = \epsilon\) we have \begin{equation}\label{r2equal}
\bar{r}_2 = r_2 ~.
\end{equation}

Using the coordinates of \(K_2\), the remaining equations from~\eqref{discrete} become

\begin{align*}
\bar{r}_2\bar{x}_2 &=r_2 x_2+ r_2^{-2} h_2 \left[r_2 x_2( r_2^2 y_2- r_2^2 x_2^2)+\lambda r_2^4 \right] ~,\\
\bar{r}_2^2\bar{y}_2 &= r_2^2 y_2+ r_2^{-2} h_2 r_2^4 ~,\\
\bar{r}_2^{-2}\bar{h}_2 &=r_2^{-2} h_2 ~,
\end{align*}

which can be simplified with \eqref{r2equal} so that we get in total
\begin{align*}
\bar{x}_2 &=x_2+h_2\left[x_2(y_2-x_2^2)+\lambda r_2\right] ~,\\
\bar{y}_2 &= y_2+h_2 ~,\\
\bar{r}_2 &=r_2 ~,\\
\bar{h}_2 &=h_2 ~.
\end{align*}

Since \(r_2\) and \(h_2\) stay constant in this chart, we can plug in the values from  $\Sigma_2^{\txtin}$ to write the maps as
 \begin{align}\label{rescalesys}
 \begin{split}
\bar{x}_2 &=x_2+\nu \delta^{\frac{1}{2}}\left(x_2(y_2-x_2^2)+\lambda \rho \delta^{\frac{1}{4}}\right) =: f(x_2, y_2) ~,\\
\bar{y}_2 &= y_2+\nu \delta^{\frac{1}{2}} ~.
 \end{split}
\end{align} 

We denote this two-dimensional map by 
\[F : \R^2 \to \R^2, ~~ F(x_2,y_2) = \left( f(x_2,y_2), ~ y_2+ \nu  \delta^{\frac{1}{2}}  \right)\,.\]
Furthermore, we abbreviate 
\[ \tilde{\lambda} :=  \lambda \rho \delta^{\frac{1}{4}}\,,  \]

such that we have \(f(x_2, y_2) = x_2 + \nu\delta^{\frac{1}{2}} \left(x_2(y_2-x_2^2)+\tilde{\lambda} \right) \). We will see that the appropriate exiting set in chart $K_2$ is given by \[ \Sigma_2^{\txtout} :=  \left\{ x_2 \in \left[  \frac{1}{4} \min(\lambda \rho,2 ) \delta^{-\frac{1}{4}},  \frac{1}{2}\delta^{-\frac{1}{4}} + \mu \right],  y_2 = \frac{1}{4}\delta ^{-\frac{1}{2}} , ~~ r_2 = \epsilon^{\frac{1}{4}}= \rho \delta^{\frac{1}{4}}, ~~ h_2 = \nu \delta^{\frac{1}{2}} \right\}. \]

Recall from chart $K_1$ the attracting center manifold $M_{\txta,1}^{0}$ (see Proposition~\ref{centermanifold_K1}). Similarly to \cite{EngelKuehn18}, this manifold corresponds with the global manifold $M_{\txta}^{0}$ on the blow-up manifold $B$. In chart $K_2$ we therefore have the attracting center manifold $M_{\txta,2}^{0} = \kappa_{12} \left(M_{\txta,1}^{0}\right)$ (with, reversely, $M_{\txta,1}^{0} = \kappa_{21} \left(M_{\txta,2}^{0}\right)$), whose behaviour is described in the following main result of this section (recall that we are still in the setting $\lambda>0$).

\begin{theorem}
\label{mainresult}
The set 
	\[ I_1:= \left[ -\frac{1}{2}\delta^{-\frac{1}{4}},  \frac{1}{2}\delta^{-\frac{1}{4}} \right] \times \{-\frac{1}{4}\delta^{-\frac{1}{2}}\} \subset \R^2\]
	gets mapped by iteration of \(F\) into the set 
	\[  I_4:=\left[  \frac{1}{4} \min(\lambda \rho,2 ) \delta^{-\frac{1}{4}},  \frac{1}{2}\delta^{-\frac{1}{4}} + \mu \right] \times \{\frac{1}{4}\delta^{-\frac{1}{2}}\}   \subset \R^2\,, \]
where \( \mu = \sqrt{2\lambda\rho}  \delta^{\frac{1}{4}} \). In particular, the transition mapping~ \(\Pi_2 : \Sigma_2^{\txtin} \to \Sigma_2^{\txtout}\) ~ is  well  defined and maps $M_{\txta,2}^{0} \cap \Sigma_2^{\txtin} $ to $M_{\txta,2}^{0} \cap \Sigma_2^{\txtout} $.
\end{theorem}

For the proof of this theorem we will analyze the evolution of our starting set at different heights (see Figure~\ref{fig:rescale}), treated in the propositions below. Before we do that, we collect some properties of the function \(f\) in the following Lemmas~\ref{lem:curve_of_fp}--\ref{lem:contraction}. 

\vspace{0.3cm}

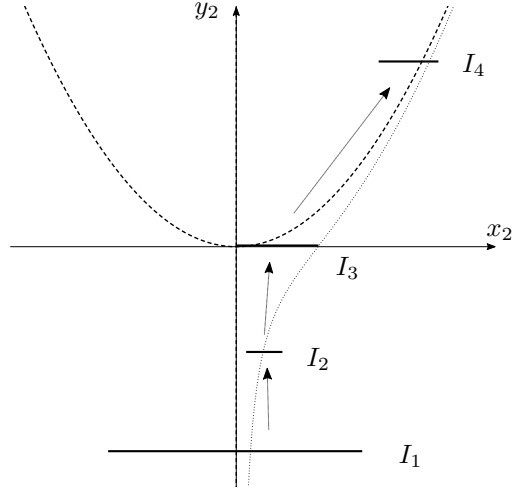
\begin{figure}[ht] 

\begin{center}
	
	\begin{overpic}[width=0.4\textwidth]{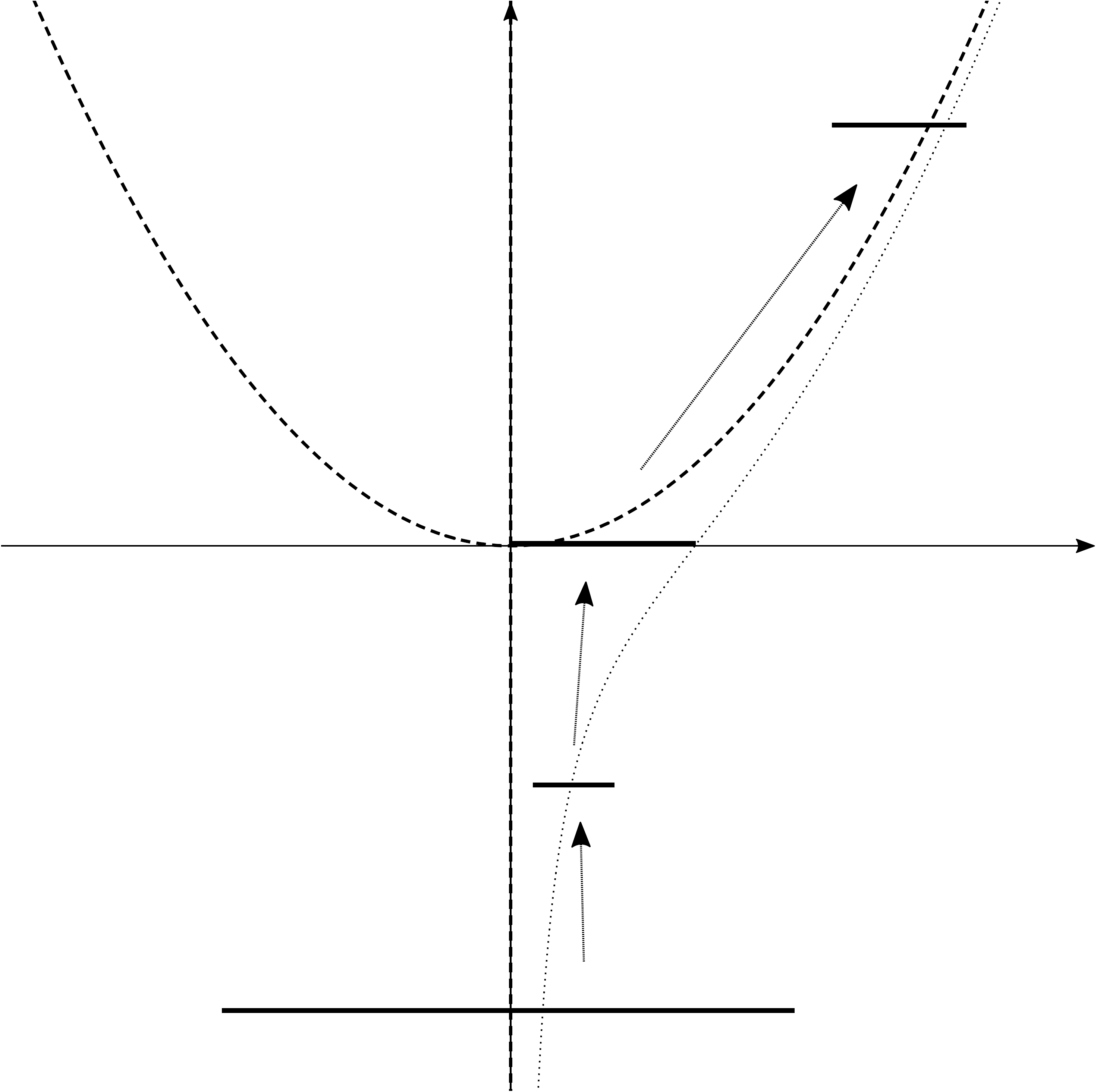}
		\footnotesize{
		\put(98,52){\(x_2\)}
		\put(38,98){\(y_2\)}
		\put(80,5){\(I_1\)}
		%\put(93,5){\(\left[ -\frac{1}{2}\delta^{-\frac{1}{4}},  \frac{1}{2}\delta^{-\frac{1}{4}} \right] \times \{-\frac{1}{4}\delta^{-\frac{1}{2}}\}\)}
		%\put(93,23){\(\left( 2 \delta^{\frac{3}{4}}\lambda\rho - \Delta(\delta), 8 \delta^{\frac{3}{4}}\lambda\rho + \Delta(\delta) \right) \times \{-\frac{1}{8}\delta^{-\frac{1}{2}}\}\)}
		\put(61,25){\(I_2\)}
		%\put(93,41){\(\left( 2 \delta^{\frac{3}{4}}\lambda\rho - \Delta(\delta), 8 \delta^{\frac{3}{4}}\lambda\rho + \Delta(\delta) \right) \times \{-\frac{1}{8}\delta^{-\frac{1}{2}}\}\)}
		\put(67,44){\(I_3\)}
		\put(93,86){\(I_4\)}
		%\put(93,86){\(\left[  \frac{1}{4} \min(\lambda \rho,2 ) \delta^{-\frac{1}{4}},  \frac{1}{2}\delta^{-\frac{1}{4}} + \mu \right] \times \{\frac{1}{4}\delta^{-\frac{1}{2}}\}\)}
	}		
	\end{overpic}
	
	\caption{Path from  \( I_1 \)  to \(  I_4 \) (Theorem~\ref{mainresult}) via $I_1$ to $I_2$ (Proposition~\ref{prop:I1_to_I2}), $I_2$ to $I_3$ (Proposition~\ref{prop:I2_to_I3}) and $I_3$ to $I_4$ (Proposition~\ref{prop:I3_to_I4}).}\label{fig:rescale}
	
	\end{center}
\end{figure}

Firstly, we characterize the positive fixed points of $f(\cdot, y_2)$ on $y_2$-fibres.

\begin{lemma}\label{lem:curve_of_fp}
	
For any $y_2 \in \mathbb{R}$, the mapping \(x_2 \mapsto f(x_2, y_2) \) has precisely one positive fixed point \(x_2^*(y_2)\), satisying the equation
	\begin{equation}\label{fixedpoint}
	 y_2 = x_2^*(y_2)^2 - \frac{\tilde{\lambda}}{x_2^*(y_2)} .
	\end{equation}
	As long as \(y_2 \in [-\frac{1}{4}\delta^{-\frac{1}{2}} , \frac{1}{4}\delta^{-\frac{1}{2}} ]\), the family of fixed points \(x_2^*(y_2)\) is monotonically increasing and satisfies \(x_2^*(y_2) < \frac{1}{2}\delta^{-\frac{1}{4}} +1 \) .
\end{lemma}

\begin{proof}

Fixed points of \(f(\cdot, y_2)\) are characterized by the equation 
\[x_2(y_2-x_2^2)+\tilde{\lambda} = 0 ~.  \] 
As \(x_2 = 0 \) is not a solution, we can solve this for \(y_2\) and obtain
  \[  y_2 = x_2^2 - \frac{\tilde{\lambda}}{x_2} ~. \]

\begin{figure}[ht]
	\begin{center}
		\begin{overpic}[width=0.5\textwidth]{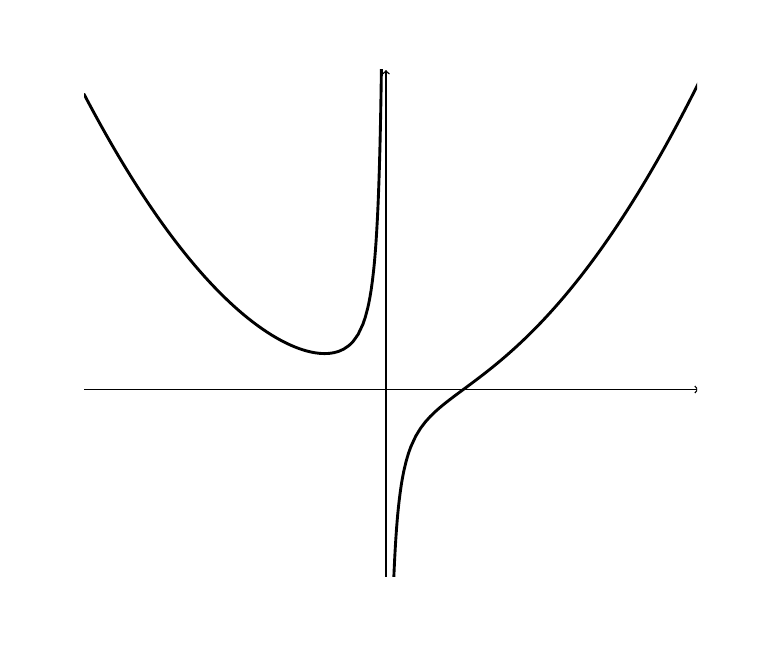}
		\put(90,32){\(x_2\)}
		\put(50,75){\(y_2\)}
		\put(70,65){\(x_2^*(y_2)\)}	
		\end{overpic}
		\caption{The curves of fixed points of \(f(\cdot, y_2)\) in the {\((x_2,y_2)\)}-plane.}
		\label{fig:fixedpoints}
	\end{center}
\end{figure}

The graph of \(x_2 \mapsto x_2^2 - \frac{\tilde{\lambda}}{x_2}\) consists of two branches. Notice that \(x_2^2 - \frac{\tilde{\lambda}}{x_2} > 0  \) for \(x_2 < 0 \). For \(x_2 >0 \) the graph is monotonically increasing (since the summands \(x_2^2\) and \(- \frac{\tilde{\lambda}}{x_2}\) are). Thus we have a curve of positive fixed points, that we can also parametrize by \(y_2\), call them \(x_2^*(y_2)\). It remains to verify the bound of \(x_2^*(y_2)\). Using \(\tilde{\lambda} < 1\), we have

\[ \left( \frac{1}{2}\delta^{-\frac{1}{4}} +1  \right)^2 - \tilde{\lambda} \left(\frac{1}{2}\delta^{-\frac{1}{4}} +1\right)^{-1} > \frac{1}{4}\delta^{-\frac{1}{2}} + 1 - 1 =  \frac{1}{4}\delta^{-\frac{1}{2}}.\]

Consequently the point
\( \left(\frac{1}{2}\delta^{-\frac{1}{4}}+1,\frac{1}{4}\delta^{-\frac{1}{2}}\right)\) 
lies below the graph of \(x_2 \mapsto x_2^2 - \frac{\tilde{\lambda}}{x_2} \), so right of the curve  \(\{ \left(x_2^*(y_2), y_2 \right)  \}\) of fixed points. Since the curve is increasing, all other fixed points \(x_2^*(y_2) \) with \(y_2 \leq \frac{1}{4}\delta^{-\frac{1}{2}} \) also satisfy the estimate \(x_2^*(y_2) \leq \frac{1}{2} \delta^{^-\frac{1}{4}} + 1\). 
\end{proof}

Secondly, we find $f(\cdot, y_2)$ to be monotonically increasing on a suitable interval and, by using Lemma~\ref{lem:curve_of_fp}, we find invariant sets under $f(\cdot, y_2)$.

\begin{lemma}
\label{lem:increasing_invariance}
Let \(y_2 \in [-\frac{1}{4}\delta^{-\frac{1}{2}} , \frac{1}{4}\delta^{-\frac{1}{2}} ]\). Then the following holds:
\begin{enumerate}
\item The function \(f(\cdot, y_2) \)  is monotonically increasing on  \([-\frac{1}{2}\delta^{-\frac{1}{4}}-1, \frac{1}{2}\delta^{-\frac{1}{4}}+1]\) for all \(\nu\) sufficiently small. \label{lem:increasing}
\item The set  \( [0, \frac{1}{2}\delta^{-\frac{1}{4}}+1]\)  is (positively) invariant under \(f(\cdot, y_2)\).
Furthermore,  when \(y_2 \leq 0
\), the set
\([-\frac{1}{2}\delta^{-\frac{1}{4}}, \frac{1}{2}\delta^{-\frac{1}{4}}]\) is (positively) invariant under \(f(\cdot, y_2)\). \label{lem:invariant}
\end{enumerate}
\end{lemma}
	
\begin{proof}
We compute the derivative \[ \tfrac{\partial}{\partial x_2} f(x_2, y_2) = 1 + \nu\delta^{\frac{1}{2}}y_2 - 3 \nu \delta^{\frac{1}{2}} x_2^2 ~,\] 
which shows that the cubic function \(f(\cdot, y_2)\) may have two stationary points located at \(x_2 = \pm \sqrt{\frac{1}{3}(\nu^{-1} \delta^{-\frac{1}{2}} +y_2 ) }\)  and then is monotonically increasing in between these. 

By choosing \(\nu\) sufficiently small, we can achieve that  the stationary points exist for all \(y_2 \in [-\frac{1}{4}\delta^{-\frac{1}{2}} , \frac{1}{4}\delta^{-\frac{1}{2}} ]\) and that \( \sqrt{\frac{1}{3}(\nu^{-1} \delta^{-\frac{1}{2}} +y_2 )} \geq \frac{1}{2} \delta^{-\frac{1}{4}} + 1 \). The choice of \(\nu\) can be made independently of \(\delta\), for all arbitrarily small \(\delta\). E. g. for \(\delta \leq 1 \) the inequality is fulfilled for all \(\nu \leq \frac{2}{15}\).
Hence, the first claim follows.

Since we have precisely one positive fixed point by Lemma~\ref{lem:curve_of_fp}, the graph of the continuous function \(f(\cdot, y_2)\) hits the diagonal precisely once. As \(f(0, y_2) = \nu \delta^{\frac{3}{4}}\lambda\rho >0 \) and \(\lim_{x_2 \to \infty} f(x_2,y_2) = -\infty\) we can conclude that it crosses the diagonal at \(x_2^*\) and  \(f(\cdot, y_2)\) lies above the diagonal (i.e. \(f(x_2,y_2) > x_2\)) for \(x_2 \in [0,x_2^*(y_2))\) and below the diagonal (i.e. \(f(x_2,y_2) < x_2\)) for  \(x_2 \in (x_2^*(y_2), \infty)\). If \(y_2 \leq 0\) we can widen the estimate \(f(x_2,y_2) < x_2\) to \(x_2 \in \left(-\infty, x_2^*(y_2)\right)\), because there is overall only one fixed point.

\begin{figure}[ht]
	\begin{center}
		\begin{overpic}[width=0.45\textwidth]{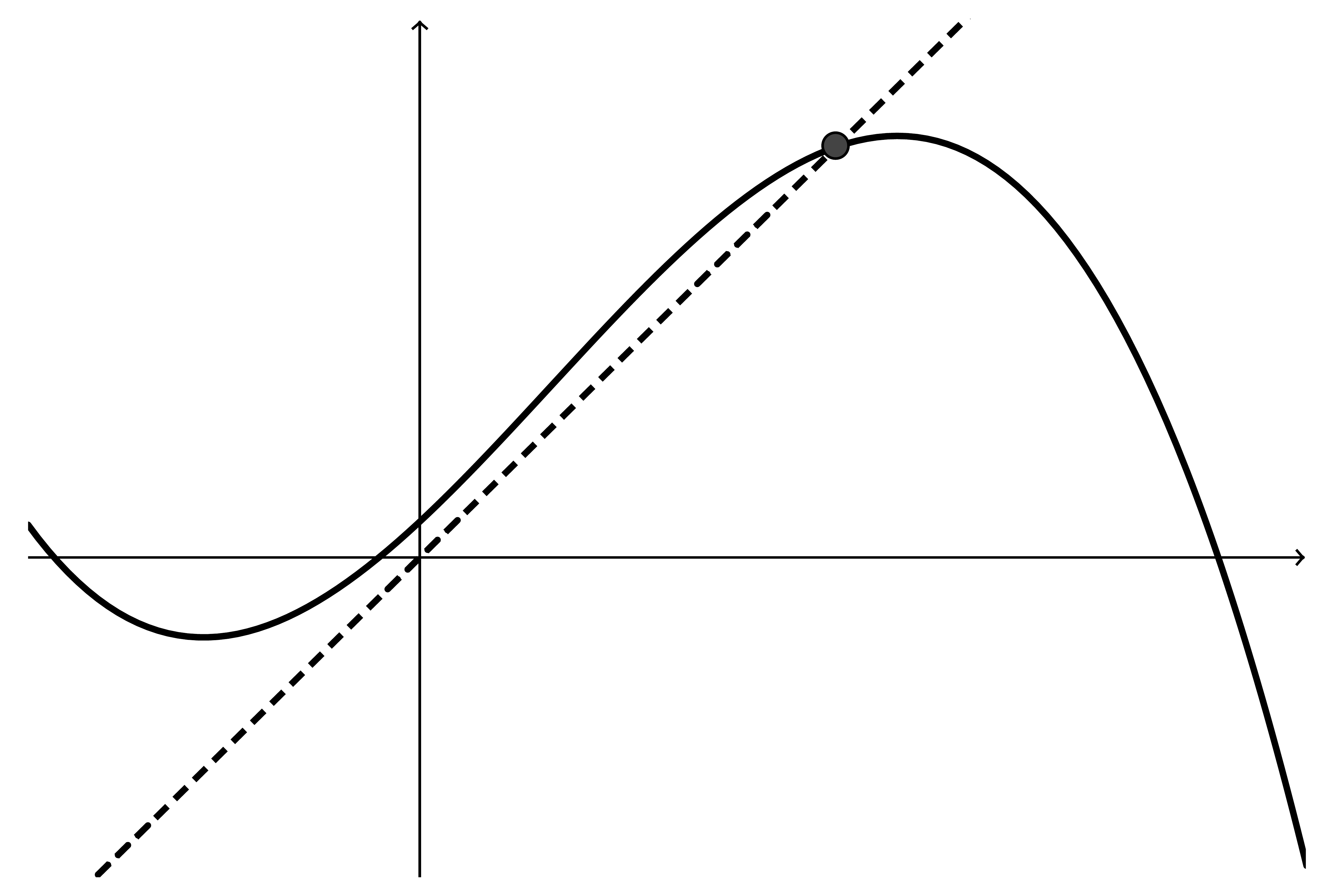}
			\footnotesize
			\put(96,28){\(x_2\)}
			
		\end{overpic}
		
		\caption{The graph of \(f(\cdot, y_2)\) together with the diagonal \(x_2 \mapsto x_2\) and the fixed point \(x_2^*(y_2)\).}
		\label{fig:stability}
	\end{center}
\end{figure}

We have already shown that \(f(\cdot, y_2)\) is increasing on \([-\frac{1}{2}\delta^{-\frac{1}{4}}-1, \frac{1}{2}\delta^{-\frac{1}{4}}+1]\). To prove the claimed invariance of \( [0, \frac{1}{2}\delta^{-\frac{1}{4}}+1]\) and  \([-\frac{1}{2}\delta^{-\frac{1}{4}}, \frac{1}{2}\delta^{-\frac{1}{4}}]\), it thus suffices to consider only the endpoints of the intervals.

Recall from Lemma~\ref{lem:curve_of_fp} that \(x_2^*(y_2) \leq \frac{1}{2} \delta^{^-\frac{1}{4}} + 1\) and, hence, we can conclude with the above that \(f(\frac{1}{2}\delta^{-\frac{1}{4}}+1, y_2) < \frac{1}{2}\delta^{-\frac{1}{4}}+1\). Since additionally \(f(0,y_2) > 0 \), the claimed positive invariance of \( [0, \frac{1}{2}\delta^{-\frac{1}{4}}+1]\) follows.

Regarding the invariance of \([-\frac{1}{2}\delta^{-\frac{1}{4}}, \frac{1}{2}\delta^{-\frac{1}{4}}]\) in case of \(y_2 \leq 0\), first note that the unique fixed point \(x_2^*(0)\) is given by \(x_2^*(0) = \tilde{\lambda}^{\frac{1}{3}} = \left( \delta^{\frac{1}{4}} \lambda \rho \right)^{\frac{1}{3}} \), which can be directly checked with the defining equation \eqref{fixedpoint}. As the curve of fixed points is increasing, we  can  estimate  \[ x_2^*(y_2) \leq x_2^*(0) = \left( \delta^{\frac{1}{4}} \lambda \rho \right)^{\frac{1}{3}} < \frac{1}{2}\delta^{-\frac{1}{4}}  \] for sufficiently small \(\delta\) and all \(y_2 \leq 0\).
With the considerations above, we thus have \[ f(-\frac{1}{2}\delta^{-\frac{1}{4}}, y_2) > -\frac{1}{2}\delta^{-\frac{1}{4}}  \text{~~~~and~~~~} 
f(\frac{1}{2}\delta^{-\frac{1}{4}}, y_2) < \frac{1}{2}\delta^{-\frac{1}{4}}~, \] yielding the invariance of the interval
\([-\frac{1}{2}\delta^{-\frac{1}{4}}, \frac{1}{2}\delta^{-\frac{1}{4}}]\).
\end{proof}

Finally, we can show the following contraction property of $f$ on relevant intervals.
\begin{lemma}\label{lem:contraction}
Let \(y_2 \in [-\frac{1}{4}\delta^{-\frac{1}{2}} , -\frac{1}{8}\delta^{-\frac{1}{2}} ]\). The function \(f(\cdot, y_2) \) restricted to the interval \([-\frac{1}{2}\delta^{-\frac{1}{4}}, \frac{1}{2}\delta^{-\frac{1}{4}}]\) is a contraction with constant \(1-\tfrac{\nu }{8} \).
\end{lemma}

\begin{proof}
Recall from the proof of Lemma~\ref{lem:increasing_invariance} that the derivative reads \[ \tfrac{\partial}{\partial x_2} f(x_2, y_2) = 1 + \nu\delta^{\frac{1}{2}}y_2 - 3 \nu \delta^{\frac{1}{2}} x_2^2 .\] 
As stated in Lemma~\ref{lem:increasing_invariance}, the map \(f(\cdot, y_2)\) is monotonically increasing on \([-\frac{1}{2}\delta^{-\frac{1}{4}}, \frac{1}{2}\delta^{-\frac{1}{4}}]\) for \(y_2 \in [-\frac{1}{4}\delta^{-\frac{1}{2}} , -\frac{1}{8}\delta^{-\frac{1}{2}} ]\), and, hence, the derivative \(\frac{\partial}{\partial x_2}f(x_2, y_2)\) is non-negative for all \((x_2,y_2) \in E := [-\frac{1}{2}\delta^{-\frac{1}{4}}, \frac{1}{2}\delta^{-\frac{1}{4}}] \times [-\frac{1}{4}\delta^{-\frac{1}{2}} , -\frac{1}{8}\delta^{-\frac{1}{2}} ] \).
Thus we obtain 
\begin{align*}
\max_{(x_2,y_2) \in E} |\tfrac{\partial}{\partial x_2} f(x_2, y_2)|  &=  \max_{(x_2,y_2) \in E} \tfrac{\partial}{\partial x_2} f(x_2, y_2)  \\
&=  \max_{(x_2,y_2) \in E}  1 + \nu\delta^{\frac{1}{2}} y_2 - 3\nu\delta^{\frac{1}{2}} x_2^2 ~ \leq 1-\tfrac{1}{8} \nu\delta^{\frac{1}{2}}\delta^{-\frac{1}{2}} = 1-\tfrac{\nu}{8} ~.
\end{align*}
Hence, the claim follows by a standard application of the mean-value theorem.
\end{proof}

We now turn to showing the transitions from $I_i$ to $I_{i+1}$ for $i=1,2,3$ (see Figure~\ref{fig:rescale}). Each transition is formulated in one of the following Propositions.
\begin{proposition}\label{prop:I1_to_I2}
The set \( I_1 = \left[ -\frac{1}{2}\delta^{-\frac{1}{4}},  \frac{1}{2}\delta^{-\frac{1}{4}} \right] \times \{-\frac{1}{4}\delta^{-\frac{1}{2}}\}\) gets mapped by iterations of \eqref{rescalesys} into the set \(I_2 := \left[ 2 \delta^{\frac{3}{4}}\lambda\rho - \Delta(\delta), 8 \delta^{\frac{3}{4}}\lambda\rho + \Delta(\delta) \right] \times \{-\frac{1}{8}\delta^{-\frac{1}{2}}\}\),  where \( \Delta(\delta) := \frac{1}{2} \delta^{-\frac{1}{4}} \rme^{-\frac{1}{64}\delta^{-1}}\).
\end{proposition}

\begin{proof}

It takes \( \frac{1}{8}\delta^{-\frac{1}{2}}\cdot (\nu \delta^{\frac{1}{2}})^{-1}  =  \frac{1}{8} \nu^{-1} \delta^{-1} \) iterations to get from \(y_2 = -\frac{1}{4}\delta^{-\frac{1}{2}}\) to \(y_2 = -\frac{1}{8}\delta^{-\frac{1}{2}}\) with steps of size \(\nu \delta^{\frac{1}{2}}\). 
Since \([-\frac{1}{2}\delta^{-\frac{1}{4}}, \frac{1}{2}\delta^{-\frac{1}{4}}]\) is invariant under \(f(\cdot, y_2)\) for all \(y_2 \in [ -\frac{1}{4}\delta^{-\frac{1}{2}}, -\frac{1}{8}\delta^{-\frac{1}{2}}]\)  by Lemma \ref{lem:increasing_invariance} and \(f(\cdot, y_2)\) is a contraction with constant \(1-\frac{\nu}{8}\) by Lemma \ref{lem:contraction}, the image of $I_1$ under iterations of~\eqref{rescalesys} has a width of at most \(\delta^{-\frac{1}{4}} \left( 1-\frac{\nu}{8}\right)^{\frac{1}{8}\nu^{-1} \delta^{-1}} \).
Since the exponential limit \[ \lim_{\nu \to 0} \left(1-\tfrac{\nu}{8}\right)^{\frac{1}{8}\nu^{-1} \delta^{-1}}  =  \rme^{-\frac{1}{64}\delta^{-1}} \] is attained monotonically from below, we can bound this width via 
\[ \frac{1}{2}\delta^{-\frac{1}{4}}  \left(1-\tfrac{\nu}{8}\right)^{\frac{1}{8}\nu^{-1} \delta^{-1}} \leq \frac{1}{2}\delta^{-\frac{1}{4}} \rme^{-\frac{1}{64}\delta^{-1}}  =:\Delta(\delta) . \]

Hence, any trajectory starting in \( I_1\) goes through the set
$$ [x_{2, \min}^*- \Delta(\delta), x_{2, \max}^*+ \Delta(\delta) ] \times \left\{-\frac{1}{8}\delta^{-\frac{1}{2}}\right\},$$
where \(x_{2, \min}^*\) and \(x_{2, \max}^*\) denote the minimal and the maximal value, respectively, that the fixed point \(x_2^*(y_2)\) attains while \(y_2\) varies over \([-\frac{1}{4}\delta^{-\frac{1}{2}}, -\frac{1}{8}\delta^{-\frac{1}{2}}]\).
Since, according to Lemma~\ref{lem:curve_of_fp}, the curve of fixed points is increasing, we have \( x_{2, \min}^*= x_2^*(-\frac{1}{4}\delta^{-\frac{1}{2}})\) and \( x_{2, \max}^*= x_2^*(-\frac{1}{8}\delta^{-\frac{1}{2}})\).
 
It remains to verify that \(x_2^*(-\frac{1}{4}\delta^{-\frac{1}{2}}) \geq  2 \delta^{\frac{3}{4}}\lambda\rho \) and \( x_2^*(-\frac{1}{8}\delta^{-\frac{1}{2}}) \leq  8 \delta^{\frac{3}{4}}\lambda\rho \).
To that purpose, we plug these bounds into equation \eqref{fixedpoint}, satisfied by the fixed points, which gives
\[  (2 \delta^{\frac{3}{4}}\lambda\rho)^2 - \frac{\lambda\rho \delta^{\frac{1}{4}}}{2 \delta^{\frac{3}{4}}\lambda\rho}  =  (2 \delta^{\frac{3}{4}}\lambda\rho)^2  -\frac{1}{2}\delta^{-\frac{1}{2}} \leq -\frac{1}{4}\delta^{-\frac{1}{2}} ~, \]
where the last inequality holds for all sufficiently small \(\delta\), and 
\[  (8 \delta^{\frac{3}{4}}\lambda\rho)^2 - \frac{\lambda\rho \delta^{\frac{1}{4}}}{8 \delta^{\frac{3}{4}}\lambda\rho}  \geq  - \frac{\lambda\rho \delta^{\frac{1}{4}}}{8 \delta^{\frac{3}{4}}\lambda\rho}  = -\frac{1}{8}\delta^{-\frac{1}{2}} ~.\] 
Since the curve of fixed points is increasing (see Figure~\ref{fig:fixedpoints}), the claim follows.
\end{proof}

\begin{proposition}\label{prop:I2_to_I3}
The set \(I_2\) gets mapped by iterations of~\eqref{rescalesys}  into  the set \(I_3:= \left[ 0, \tilde{\lambda}^{\frac{1}{3}} \right] \times \{0 \}\). 
\end{proposition}

\begin{proof}
Recall that the fixed point \(x_2^*(0)\) is given by \(x_2^*(0) = \tilde{\lambda}^{\frac{1}{3}} = \left( \delta^{\frac{1}{4}} \lambda \rho \right)^{\frac{1}{3}} \). 
Furthermore, we have
 \[\left[ 2 \delta^{\frac{3}{4}}\lambda\rho - \Delta(\delta), 8 \delta^{\frac{3}{4}}\lambda\rho + \Delta(\delta) \right] \subset \left[0, \left( \delta^{\frac{1}{4}} \lambda \rho \right)^{\frac{1}{3}} \right] \] 
for all sufficiently small \(\delta\).
Analogously to Lemma \ref{lem:contraction}, we observe that, for all  \(y_2 \in [-\frac{1}{8}\delta^{-\frac{1}{2}}, 0)\), the map \(f(\cdot , y_2)\) is a contraction on \(\left[ 0, \tilde{\lambda}^{\frac{1}{3}} \right]\)  to the fixed point \(x_2^*(y_2)\).
Since \(x_2^*(y_2) \in \left[ 0, \tilde{\lambda}^{\frac{1}{3}} \right] \) for all \(y_2< 0 \), we can deduce that \(\left[ 0, \tilde{\lambda}^{\frac{1}{3}} \right]\) is invariant under \(f(\cdot , y_2)\) for all  \(y_2 \in \left[-\frac{1}{8}\delta^{-\frac{1}{2}}, 0\right)\). This implies the claim.
%Consequently the image of \(\left( 2 \delta^{\frac{3}{4}}\lambda\rho - \Delta(\delta), 8 \delta^{\frac{3}{4}}\lambda\rho + \Delta(\delta) \right) \times \{-\frac{1}{8}\delta^{-\frac{1}{2}}\}\) at height \(y_2 = 0\) lies within \(\left( 0, \tilde{\lambda}^{\frac{1}{3}} \right) \times \{0 \}\). 	
\end{proof}
The final transition is described in the following proposition.

\begin{proposition}\label{prop:I3_to_I4}
The set \(I_3\) gets mapped by iterations of \eqref{rescalesys} into the set \[ I_4 = \left[  \frac{1}{4} \min(\lambda \rho,2 ) \delta^{-\frac{1}{4}},  \frac{1}{2}\delta^{-\frac{1}{4}} + \mu \right] \times \left\{\frac{1}{4}\delta^{-\frac{1}{2}}\right\}\,,\] 
where \( \mu = \sqrt{2\lambda\rho}  \delta^{\frac{1}{4}} \).
\end{proposition}

\begin{proof} 

We first consider the set $S_1$ (see Figure~\ref{fig:Si}), defined by
\[S_1:= \left\{ (x_2,y_2) ~ \vert ~ 0 \leq x_2 \leq x_2^*(y_2), ~  0\leq y_2  \leq \frac{1}{4} \delta^{-\frac{1}{2}} \right\}\,.\]
\begin{figure}[ht]
	\begin{center}
		
		\begin{overpic}[width=0.30\textwidth]{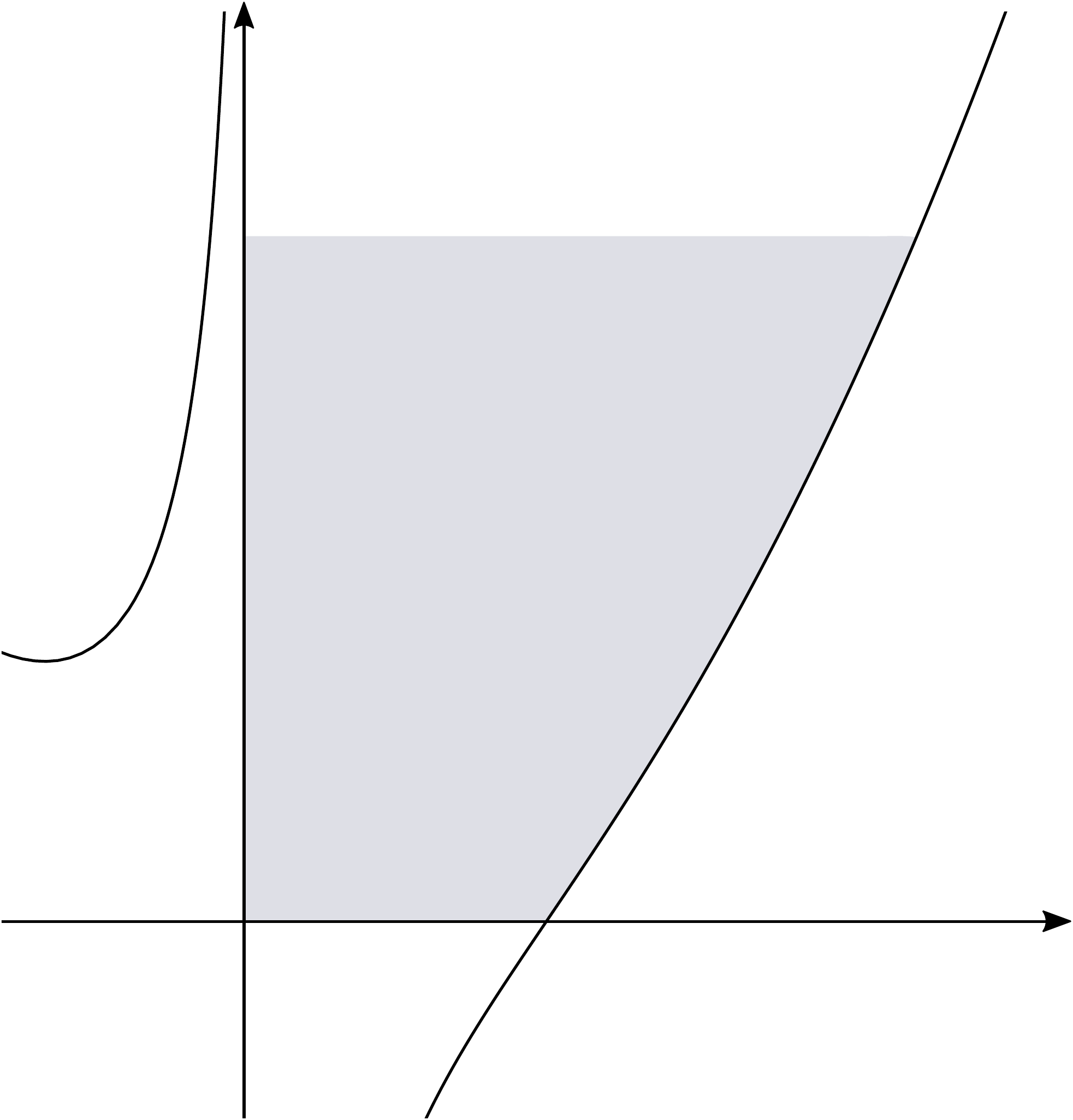}
			\put(50,106){\(S_1\)}
			\footnotesize
			\put(92,21){\(x_2\)}
			\put(24,98){\(y_2\)}
			
		\end{overpic}	
		\hspace{2cm}
		\begin{overpic}[width=0.30\textwidth]{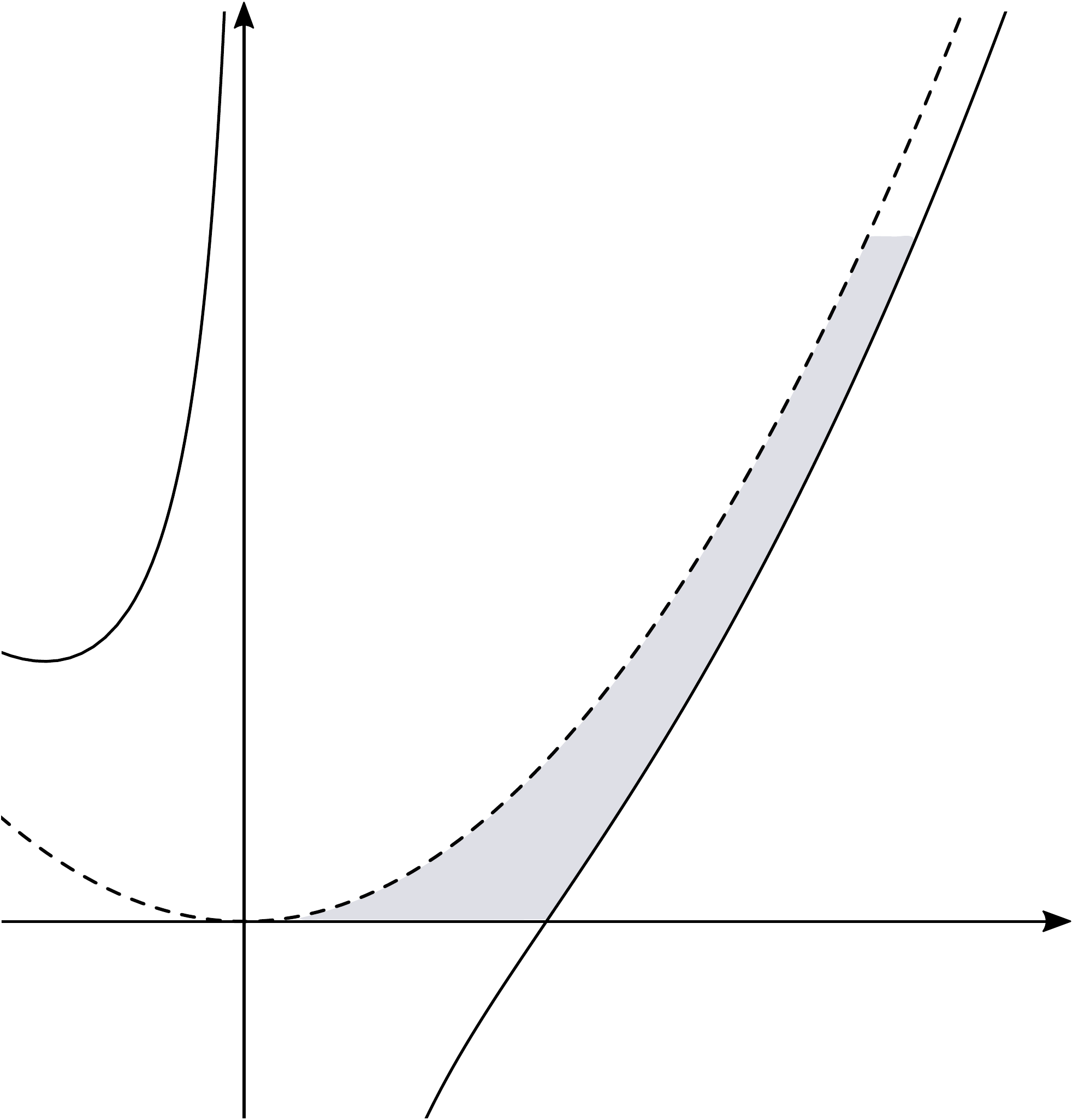}
			\put(50,106){\(S_2\)}
			\footnotesize
			\put(92,21){\(x_2\)}
			\put(24,98){\(y_2\)}
		\end{overpic}	
		\\ \vspace{1.5cm}
		\begin{overpic}[width=0.30\textwidth]{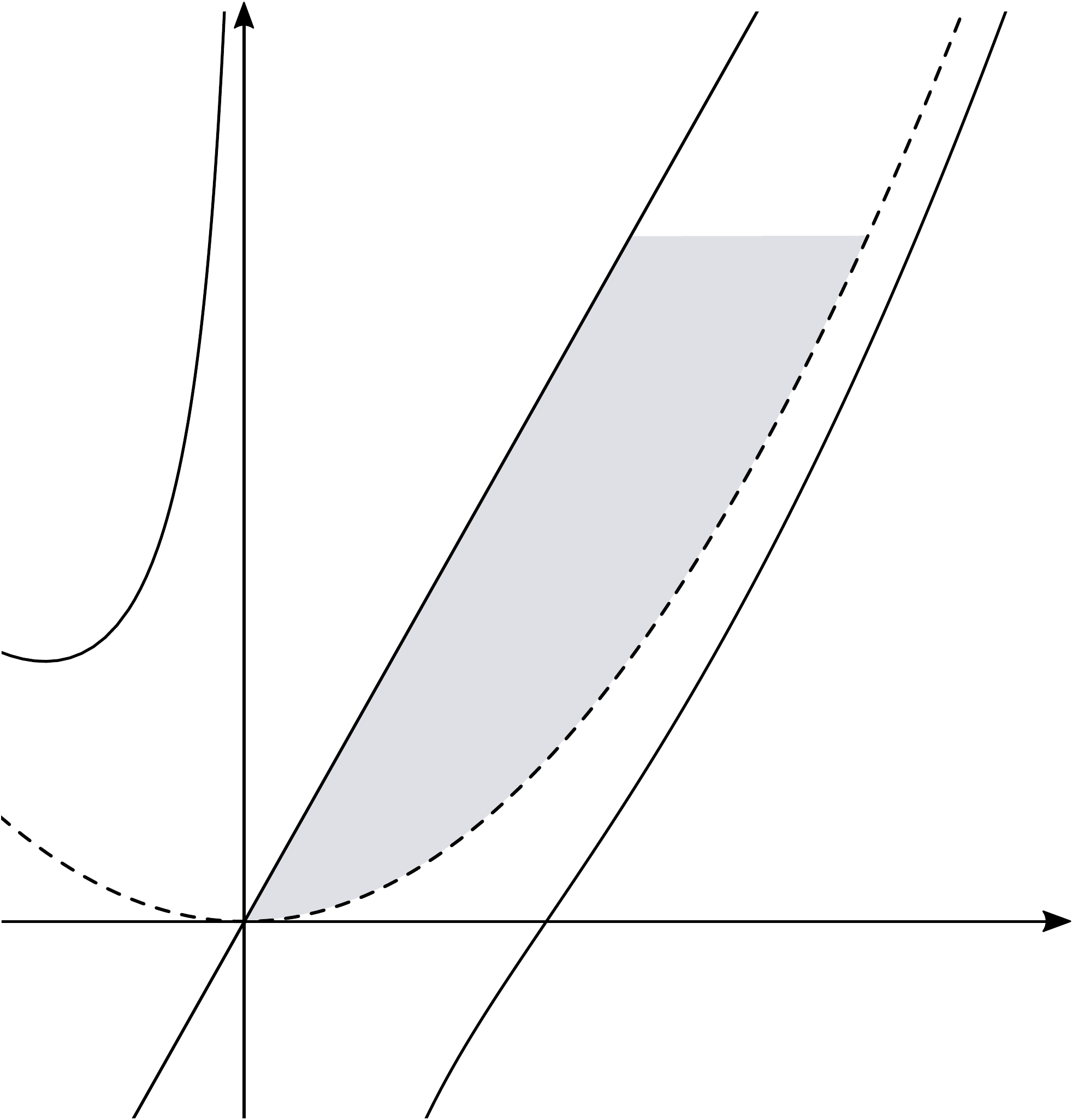}
			\put(40,106){\(S_3 ~ ( \text{for} ~ \lambda \rho < 2)\)}
			\footnotesize
			\put(92,21){\(x_2\)}
			\put(24,98){\(y_2\)}
		\end{overpic}
		\hspace{2cm}
		\begin{overpic}[width=0.30\textwidth]{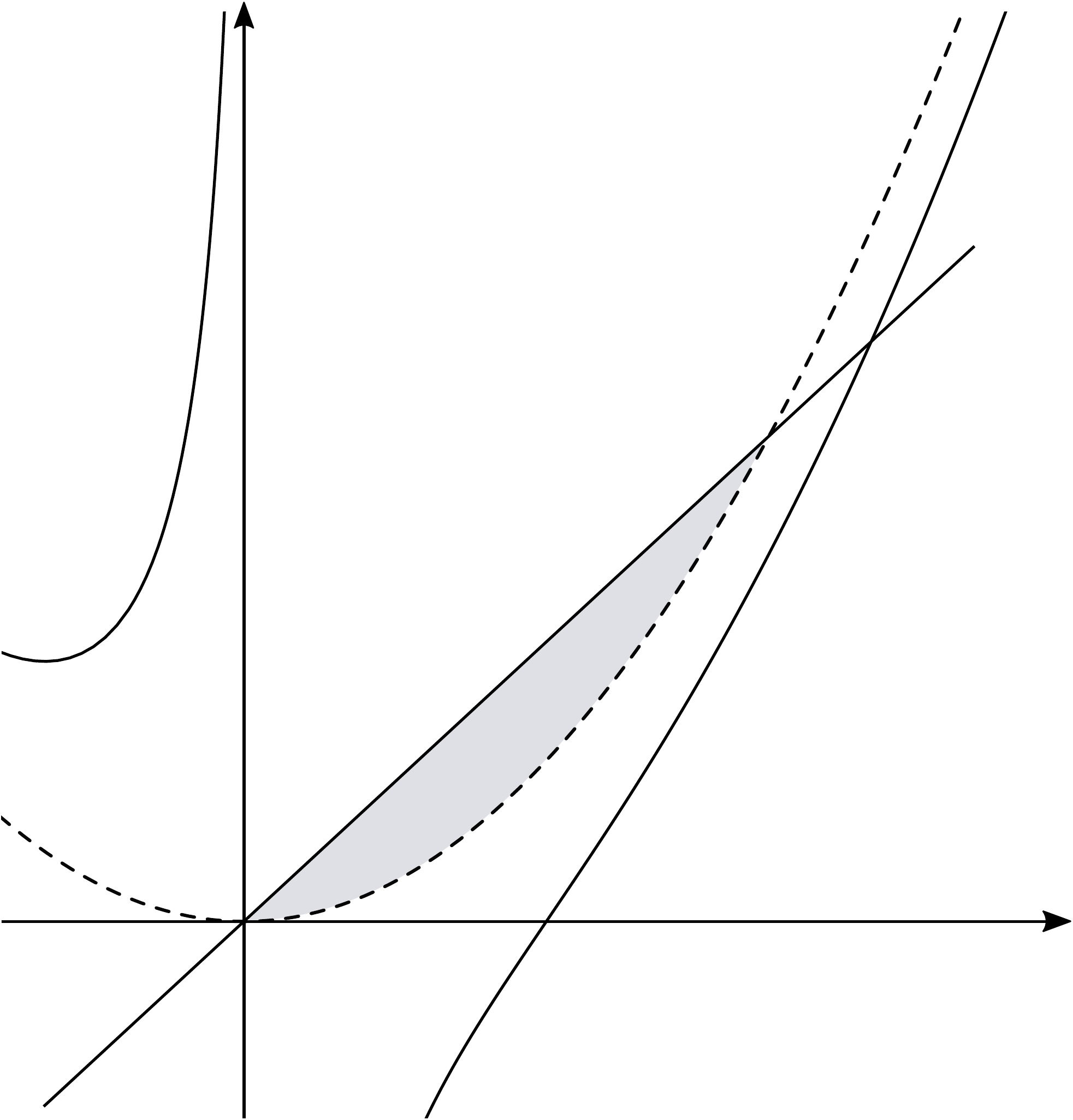}
			\put(40,106){\(S_3 ~ (\text{for} ~ \lambda \rho \geq 2)\)}
			\footnotesize
			\put(92,21){\(x_2\)}
			\put(24,98){\(y_2\)}
		\end{overpic}
		
		\caption{The sets \(S_1, S_2\) and \(S_3\) in the \((x_2,y_2)\)-plane. The shape of \(S_3\) depends on the size of \(\lambda\rho\).}

		\label{fig:Si}
	\end{center}
\end{figure}
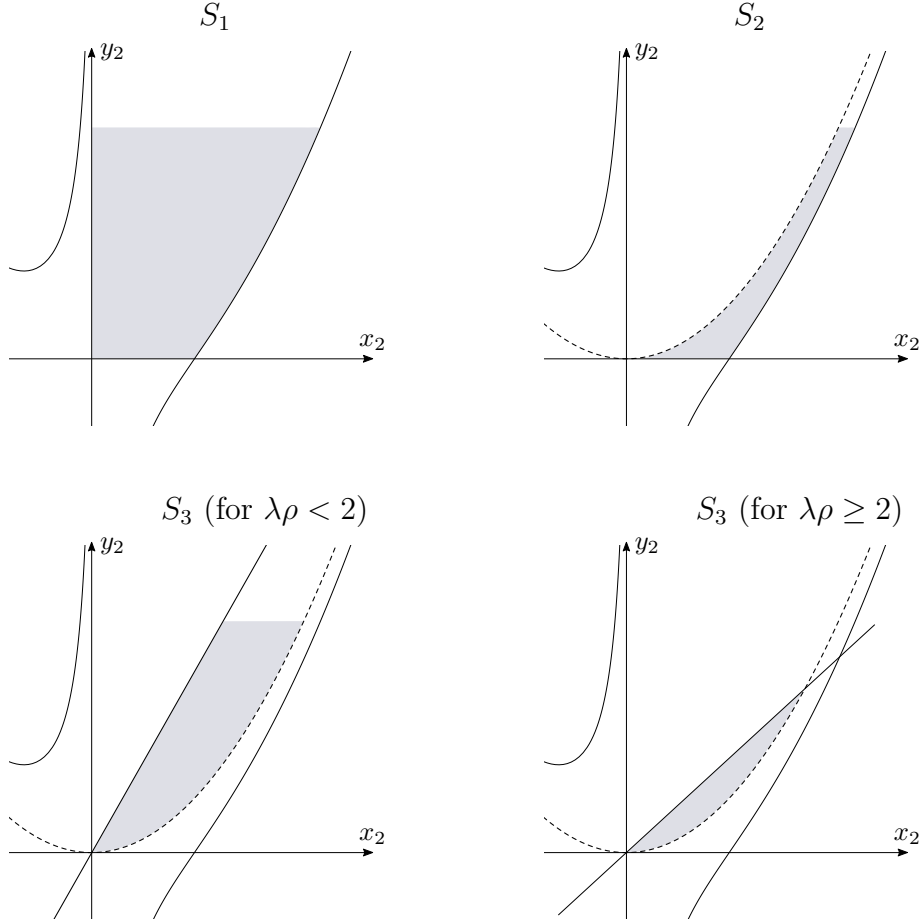
Let \((x_{2,0}, y_{2,0})\) be in \(S_1\).  Now, since \(f(\cdot, y_{2,0})\) is increasing and also the curve of fixed points is increasing, we have 
\[\bar{x}_{2,0} = f(x_{2,0},y_{2,0}) \leq x_2^*(y_{2,0}) \leq x_2^*(y_{2,0} + \nu \delta^{\frac{1}{2}} ) = x_2^*(\bar{y}_{2,0})\,.\]
Hence, the point \((\bar{x}_{2,0}, \bar{y}_{2,0})\) lies in \(S_1\), as long as $y_{2,0} \leq \frac{1}{4} \delta^{-\frac{1}{2}} - \nu \delta^{\frac{1}{2}}$.
This means that trajectories leave \(S_1\) through the upper boundary at \(y_2 = \frac{1}{4} \delta^{-\frac{1}{2}}\).
%is (almost) positively invariant (points leave \(S_1\) through the upper boundary at \(y_2 = \frac{1}{4} \delta^{-\frac{1}{2}}\)).

Moreover, we consider the sets 
\[S_2 := \{ (x_2,y_2) ~ \vert ~ y_2 \leq x_2^2  \}  \cap S_1\,, \quad S_3 := \{  (x_2,y_2) ~ \vert ~ x_2^2 < y_2 \leq \tilde{\lambda}^{-1} x_2   \} \cap S_1 \,, \]
as depicted in Figure~\ref{fig:Si}. 
For points \((x_2,y_2) \in S_3\), we have \(\bar{x}_2 = f(x_2,y_2) > x_2 +  \nu \delta^{\frac{3}{4}} \lambda \rho = x_2 + \nu \delta^{\frac{1}{2}} \tilde{\lambda} \) as well as \(\bar{y}_2 = y_2 + \nu \delta^{\frac{1}{2}}\). Thus they will be mapped into \(S_2 \cup S_3\), as long as $y_{2} \leq \frac{1}{4} \delta^{-\frac{1}{2}} - \nu \delta^{\frac{1}{2}}$. Points on the left-hand boundary of \(S_2\), given by \(\{y_2=x_2^2\}\),  clearly satisfy the same estimate such that they are also mapped into \(S_2 \cup S_3\). By monotonicity of \(f(\cdot, y_2)\), see Lemma~\ref{lem:curve_of_fp}, this holds for all points in \(S_2\).
Hence, we deduce that trajectories leave \(S_2 \cup  S_3\) through the upper boundary at \(y_2 = \frac{1}{4} \delta^{-\frac{1}{2}}\).

In order to find a right-hand bound for the set $S_1$, we
%We next will bound the \(x_2\)-coordinate of the  fixed point \(x_2^*(\frac{1}{4}\delta^{-\frac{1}{2}})\).
%Define \( \mu = \sqrt{2\tilde{\lambda}} ~ \delta^{\frac{1}{8}} = \sqrt{2\lambda\rho} ~ \delta^{\frac{1}{4}} \)
%and 
compute 
\[ \left(\frac{1}{2}\delta^{-\frac{1}{4}} + \mu \right)^2 - \tilde{\lambda} \left(\frac{1}{2}\delta^{-\frac{1}{4}} + \mu\right)^{-1} \geq
\frac{1}{4}\delta^{-\frac{1}{2}}
+ \mu^2 - 2\tilde{\lambda}\delta^{\frac{1}{4}} = \frac{1}{4}\delta^{-\frac{1}{2}},
\]
where we used that \( \mu = \sqrt{2\tilde{\lambda}} ~ \delta^{\frac{1}{8}} = \sqrt{2\lambda\rho} ~ \delta^{\frac{1}{4}} \).
Since the curve of fixed points is increasing (see Figure~\ref{fig:fixedpoints}), this means that we can estimate  \(x_2^*(\frac{1}{4}\delta^{-\frac{1}{2}}) \leq \frac{1}{2}\delta^{-\frac{1}{4}} + \mu\). We observe from the proof of Proposition~\ref{prop:I2_to_I3} that
\(\left( 0, \tilde{\lambda}^{\frac{1}{3}} \right) \times \{0 \}\) is a subset of \(S_2 \cup S_3\). Furthermore, it is now easy to see that
\[(S_2 \cup S_3 ) \cap \left(\R \times \left\{ \frac{1}{4} \delta^{-\frac{1}{2}} \right\} \right) \subseteq
 \left[  \frac{1}{4}\lambda \rho \delta^{-\frac{1}{4}},  \frac{1}{2}\delta^{-\frac{1}{4}} + \mu \right] \times \left\{ \frac{1}{4}\delta^{-\frac{1}{2}}\right\}\,, \ \text{ when } \lambda \rho < 2\,, \] 
and 
\[(S_2 \cup S_3 ) \cap \R \times \left\{ \frac{1}{4} \delta^{-\frac{1}{2}} \right\}  \subseteq
\left[  \frac{1}{2} \delta^{-\frac{1}{4}},  \frac{1}{2}\delta^{-\frac{1}{4}} + \mu \right] \times \left\{\frac{1}{4}\delta^{-\frac{1}{2}}\right\}\,, \ \text{ when } \lambda \rho \geq 2\,.\]
This concludes the proof.
\end{proof} 
Theorem \ref{mainresult} is now an immediate consequence of combining 	\Cref{prop:I1_to_I2}, \Cref{prop:I2_to_I3} and \Cref{prop:I3_to_I4}.

\subsection{Dynamics in the exiting chart}
\label{sec:exiting}

Transforming \(\Sigma_2^{\txtout}\) to the coordinates of the third chart gives 
\[\kappa_{23}(\Sigma_2^{\txtout}) =  \left\{x_3 \in \left[  \min(\tfrac{1}{2}\lambda \rho,1 ), 1+ 2\mu \delta^{\frac{1}{4}} \right]~,~~r_3 = \tfrac{\rho}{2}~,~~\epsilon_3 = 16\delta~,~~h_3 = \tfrac{\nu}{4}\right\} ~.\] 
Furthermore, we define 
\begin{equation} \label{Sigma3in}
\Sigma_3^{\txtin} := \left\{x_3 \in \left[ 1-\theta, 1+\theta \right]~,~~r_3 = \tfrac{\rho}{2}~,~~\epsilon_3 = 16\delta~,~~h_3 = \tfrac{\nu}{4}\right\} ~, 
\end{equation}
with \(\theta = \max \{1-\tfrac{\lambda \rho}{2}, \tfrac{1}{2}\}\), such that 
\(\kappa_{23}(\Sigma_2^{\txtout}) \subset \Sigma_3^{\txtin}\) for sufficiently small \(\delta\).
As already noted in section \ref{sec:entering}, we may assume, due to the controllably small error, that specific levels are hit by the trajectories and therefore, in particular, we assume \(\rho = \tilde{\rho}\), where \(\tilde{\rho} \) was used to define \(\Delta_{\textnormal{out}}^+\) \eqref{Deltaout}. 

Similarly to the situation in \(K_1\), we consider the domain 
$$D_3 := \{(x_3, r_3, \epsilon_3, h_3) \in \mathbb{R}^4 : r_3 \in 
[ 0, \rho], \epsilon_3 \in [0, 16 \delta], h_3 \in [ 0, \nu] \}$$ 
for the chart \(K_3\), and we obtain the map
\begin{align} \label{K3_dynamics}
\begin{array}{rcrcl}
&\bar{x}_3& &=&(1+h_3\epsilon_3)^{-\frac{1}{2}} \left[ x_3+ h_3\left(x_3(1-x_3^2)+\lambda r_3\epsilon_3 \right) \right]~, \\
&\bar{r}_3& &=& (1+h_3\epsilon_3)^{\frac{1}{2}}~ r_3~,\\
&\bar{\epsilon}_3& &=& (1+h_3\epsilon_3)^{-2} \epsilon_3~,\\
&\bar{h}_3& &=&(1+h_3\epsilon_3)~ h_3~,
\end{array}
\end{align}
corresponding with~\eqref{discrete}.
Similarly to the system obtained in $K_1$, the special structure of \eqref{K3_dynamics} again allows to explicitly determine solutions of the induced dynamical system in the components $r_3, \epsilon_3$ and $h_3$. 

\begin{lemma}
\label{r3h3eps3}
	For $\zeta_0 := \frac{1}{h_3(0)\epsilon_3(0)} > 0$, the trajectories of~\eqref{K3_dynamics} in $r_3, \epsilon_3, h_3$ are given by
	\begin{equation*} 
	h_3(n) = h_3(0) \frac{\zeta_0+n}{\zeta_0}~,~~ r_3(n) = r_3(0) \sqrt{ \frac{\zeta_0+n}{\zeta_0}} ~, ~~ \epsilon_3(n) = \epsilon_3(0) \frac{\zeta_0^2} {(\zeta_0+n)^2}~,
	\end{equation*}
	for $n\in \mathbb{N}$. 
\end{lemma}

\begin{proof}
	Let \(\vartheta := h_3 \epsilon_3 \). Multiplying both sides of the equations for \( \epsilon_3 \) and \(h_3\)  in \eqref{K3_dynamics} yields \[ \bar{\vartheta} = (1+\vartheta)^{-1} \vartheta ~. \] Solving this recursion for some initial condition \(  \vartheta(0)= \zeta_0^{-1} >0\) gives \[ \vartheta(n) =\frac{1}{\zeta_0 + n} ~. \]
	We use this observation to calculate \[h_3(n) = h_3(0) \prod_{k=0}^{n-1} (1+\theta(k)) =h_3(0) \prod_{k=0}^{n-1} \frac{\theta(k)}{\theta(k+1)} = h_3(0) \frac{\theta(0)}{\theta(n)} =h_3(0) \frac{\zeta_0+n}{\zeta_0},\]
	and analogously 
	\begin{equation*} 
	r_3(n) = r_3(0) \sqrt{ \frac{\zeta_0+n}{\zeta_0}} \text{~~~~ as well as ~~~} \epsilon_3(n) = \epsilon_3(0) \frac{\zeta_0^2} {(\zeta_0+n)^2}~.
	\end{equation*}
	This finishes the proof.
\end{proof}

We observe from \eqref{K3_dynamics} that the hyperplane \(\{\epsilon_3 = 0\}\) is an  invariant set for the system \eqref{K3_dynamics}, foliating into the invariant lines \( \{\epsilon_3=0, r_3=r_3^*, h_3 = h_3^*\}\)  for all $r_3^*, h_3^* \geq 0$. 
Each of these lines has three fixed points, located at \((x_3, r_3, \epsilon_3, h_3)= (s, r_3^*, 0, h_3^*)\), \( s \in \{-1, 0, 1\} \).
Linearizing  around each of these, for $h_3^* > 0$, we see that in \(x_3\)-direction the fixed point at \(x_3 = 0 \) is unstable with eigenvalue \(1+h_3^*\) while those at  \(x_3 = -1 \) and \(x_3 = 1\) are stable with eigenvalue \(1-2h_3^*\) (recall that $h_3^* \leq h \rho^2 < \frac{1}{3} <1$).

Since, in our considerations, we enter $K_3$ via $\Sigma_3^{\txtin}$, our main interest lies in the family of stable fixed points at \((x_3, r_3, \epsilon_3, h_3)= (1, r_3^*, 0, h_3^*)\),  corresponding with the branch \(S_a^+\) of the critical manifold for $r_3^* > 0$.
Each of these fixed points has a three-dimensional center eigenspace as well as a one-dimensional stable eigenspace in \(x_3\)-direction with eigenvalue \(1-2h_3^*\). The union of these fixed points forms an invariant manifold, which we call
$$ S_{\txta,3}^{+} = \{(x_3, r_3, \epsilon_3, h_3) \in D_3 \,:\, x_3 = 1,\epsilon_3 = 0. \}$$
In particular, for each $h_3 \geq 0$ it contains the fixed point
$$ p_a^+(h_3) = (1,0,0, h_3)\,,$$
which has gained hyperbolicity due to the desingularization of the origin. In analogy to Proposition~\ref{centermanifold_K1} we get the following:

\begin{proposition} 
\label{centermanifold_K3}
The invariant manifold $S_{\txta,3}^{+}$ extends to a center-stable invariant manifold \(M_{\txta,3}^+\) (at $p_a^+(0)$) which is given in $D_3$ by a graph $x_3 = l_3(r_3, \epsilon_3, h_3)$ for a smooth mapping $l_3$. Furthermore, $\epsilon_3$ is decreasing in $D_3$ (and thereby, in particular, on $M_{\txta,3}^{+}$), whereas $h_3, r_3$ are increasing in $D_3$ (and thereby on $M_{\txta,3}^{+}$).
\end{proposition}

\begin{proof}
This is an immediate consequence of the considerations above and classical center manifold theory. Similarly to the proof of Proposition~\ref{centermanifold_K1}, the second claim follows from Lemma~\ref{r3h3eps3}.
\end{proof}

Note that, on $\{r_3 > 0, \epsilon_3 > 0, h_3 > 0 \}$, the manifold \(M_{a,3}^+\) corresponds to the union of the slow manifolds \(S_{a, \epsilon, h}^+\). 

We will follow the iterations of~\eqref{K3_dynamics} until the values \(r_3 = \rho, \epsilon_3 = \delta, h_3 = \nu \) are reached. Lemma \ref{r3h3eps3} tells us that this is the case when \(\frac{\zeta_0+n^*}{\zeta_0} = 4\), which means that the number of iterations equals \(n^* = 3 \zeta_0 =\frac{3}{4 \nu \delta} \).
Let \(\Pi_3: \Sigma_3^{\txtin} \to \R^4\) be the transition map induced by \(n^* \) iterations of \eqref{K3_dynamics} and consider the set \[\Sigma_3^{\txtout} := \left\{  x_3 \in 1+ \rho^{-1} J ,  r_3= \rho, \epsilon_3 = \delta, h_3 = \nu \right\}.\]
Note that with this choice we have \(K_3(\Sigma_3^{\txtout}) =  \Delta_{\textnormal{out}}^+\). 
The following Lemma summarizes the properties of the transition map~\(\Pi_3\).

\begin{lemma} \label{lem:transition3}
For $\delta > 0$ sufficiently small, we have $\Pi_3(\Sigma_3^{\textnormal{in}}) \subset \Sigma_3^{\textnormal{out}} $, where \(\Pi_3(\Sigma_3^{\textnormal{in}})\) has a \(x_3\)-width of at most \(	\left(1-C^*\frac{\nu}{4}\right)^{\frac{3}{4 \nu \delta}} \)  for some constant \(C^*>0\). Furthermore, the intersections \(M_{\txta, 3}^+ \cap \Sigma_3^{\textnormal{in}}\) as well as \(M_{\txta, 3}^+ \cap \Pi_3(\Sigma_3^{\textnormal{in}})\) are non-empty sets.
\end{lemma}

\begin{proof}
From the explicit solutions in Lemma~\ref{r3h3eps3}, we directly see that   \[\Pi_3(\Sigma_3^{\txtin}) \subset \left\{r_3= \rho, \epsilon_3 = \delta, h_3 = \nu \right\} .\]
Next, we consider the \(x_3\)-equation of \eqref{K3_dynamics} for \(\epsilon_3 = 0 \) and $h_3 > 0$, which reads 
\[\bar{x}_3 = x_3+ h_3\left(x_3(1-x_3^2)\right) =: g(x_3, h_3)~. \]
The cubic function \(g(\cdot, h_3)\) is increasing between the stationary points at
\begin{equation} \label{hrestriction}
x_3 = \pm\sqrt{\tfrac{1+h_3}{h_3}} =:\pm \kappa ~.  
\end{equation} 
Taking \(h_3 \leq \frac{1}{3}\), we achieve that \(\kappa \geq 2 > 1 + \theta \).
Since the fixed point at \(x_3 =0\) is unstable we can deduce that the set \((0, \kappa ]  \) gets attracted to the fixed point at \(x_3 = 1\), and so does \(  \left[ 1-\theta, 1+\theta \right] \subset (0, \kappa ]\). 
	
We turn to giving estimates on the contraction rate  towards \(x_3 =1\) for different starting values in \([1-\theta, 1+\theta].\)	Note that the map \(g(\cdot, h_3)\) is increasing and concave on \( (0, \kappa] \). Since it is increasing, the subintervals \([1-\theta, 1)\) and \((1, 1+ \theta]\)  get mapped by \(g(\cdot, h_3)\) into themselves.
Due to its concavity, the function  \(g(\cdot, h_3)\) lies below the tangent at \(x_3 =1\), so that the contraction rate for values in \(\left(1, 1+ \theta \right]\) is at least as strong as the linear rate \(1-2h_3\) coming from the linearization around \(x_3 =1\). Concavity also yields that \(g(\cdot, h_3)\) lies above the secant on the interval \([1-\theta, 1]\). Thus  the slope of the corresponding secant gives an estimate for the contraction rate of points in \([1-\theta, 1)\) to the fixed point at \(x_3 = 1\).  

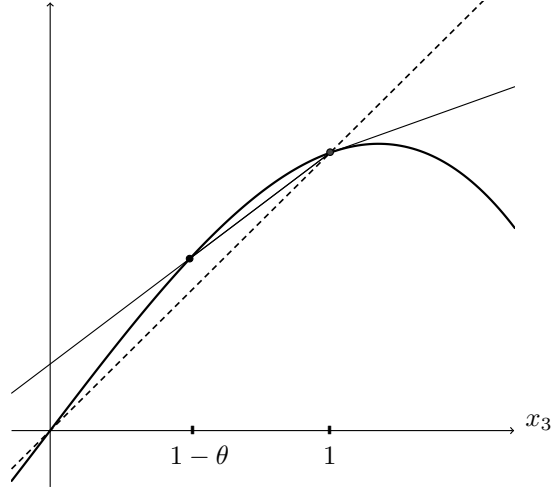
\begin{figure}[ht]
	\begin{center}
		\begin{overpic}[width=0.45\textwidth]{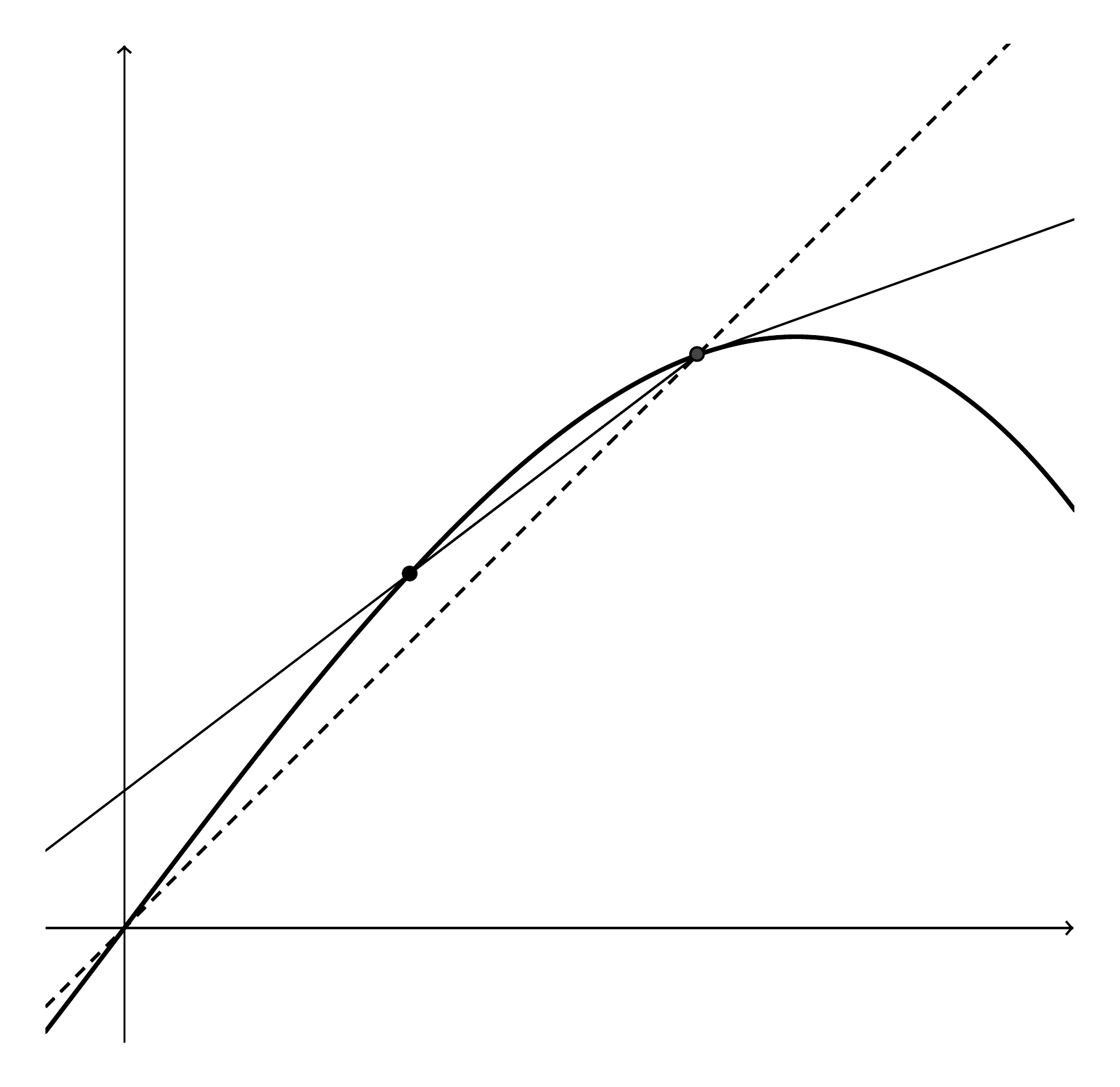}
			\footnotesize
			\put(98,15){\(x_3\)}
			\put(60,15){\(\bracevert\)}
			\put(61,8){\(1\)}
			\put(35,15){\(\bracevert\)}
			\put(33,8){\(1-\theta\)}
		\end{overpic}

		\caption{The map \(g(\cdot, h_3)\) together with the diagonal \(x_3 \mapsto x_3\) (dashed), the tangent at 1 and the secant over \([1-\theta, 1]\) .}\label{tangentsecant}
	\end{center}
\end{figure}

In more detail, linear interpolation between the points \(\left(1-\theta, g(1-\theta, h_3)\right)\) and \((1,1)\) yields the slope
	\begin{align*}
	 \frac{1- g(1-\theta, h_3)}{1- (1-\theta)}  &= \frac{1- \left[ (1-\theta) + h_3 \left( (1-\theta)(1-(1-\theta)^2)\right)  \right]}{1- (1-\theta)}  \\
	&= 1- h_3 \frac{(1-\theta)\left(1- (1-\theta)^2 \right) }{\theta } 
	= 1- h_3^* \frac{(1-\theta)\left(2\theta - \theta^2 \right) }{\theta } \\
	& = 
	1- h_3 (1-\theta)\left(2 - \theta \right) = 1- h_3 C^*(\theta) ~, 
	\end{align*}
where 	\begin{equation}\label{def:C*}
			C^*(\theta) := (1-\theta)(2-\theta)~.
		\end{equation} Note that \(0 < C^*(\theta) < 2\)	 and  \(C^*(\theta) \to 0\) as \(\theta \to 1\).

We conclude that, on the interval \([1-\theta, 1+\theta]\), the map \(g(\cdot, h_3)\) is contracting with constant \(1-h_3C^*\). The transition map \(\Pi_3\) defined on 
\[\Sigma_3^{\txtin} = \left\{x_3 \in \left[ 1-\theta, 1+\theta \right],~r_3 = \tfrac{\rho}{2},~\epsilon_3 = 16\delta,~h_3 = \tfrac{\nu}{4}\right\}\]
  is in \(x_3\)-direction a perturbation of the \(n^*\)-fold of \(g(\cdot, h_3)\), as we now have \(r_3 , \epsilon_3 > 0\).

 Also note that during the \( n^* = \frac{3}{4 \nu \delta} \) iterations that define \(\Pi_3\), the variable \(h_3\) lies in the interval \([\tfrac{\nu}{4}, \nu]\). Thus, if we choose \(\delta\) sufficiently small, we can achieve that \(\Pi_3\) is a contraction (in \(x_3\)-direction) with a rate of \(\frac{1}{2}\left(1-C^*\frac{\nu}{4}\right)^{\frac{3}{4 \nu \delta}}\).
Consequently the image  \(\Pi_3(\Sigma_3^{\textnormal{in}})\) has a \(x_3\)-width of at most \[ 2\theta \cdot \tfrac{1}{2} \left(1-C^*\tfrac{\nu}{4}\right)^{\frac{3}{4 \nu \delta}} \leq  \left(1-C^*\tfrac{\nu}{4}\right)^{\frac{3}{4 \nu \delta}} ~. \]
		
Since for \(\delta \to 0 \) the point \(M_{\txta, 3}^+  \cap \left\{r_3= \frac{\rho}{2}, \epsilon_3 = 16\delta, h_3 = \frac{\nu}{4} \right\}\) approaches \((1,\frac{\rho}{2}, 0, \frac{\nu}{4})\), so especially its \(x_3\) coordinate approaches \(1\), we can deduce that \(M_{\txta, 3}^+  \cap \left\{r_3= \frac{\rho}{2}, \epsilon_3 = 16\delta, h_3 = \frac{\nu}{4} \right\}\) lies in \(\Sigma_3^{\txtin}\) for \(\delta\) small enough. Moreover, since the manifold \(M_{\txta, 3}^+ \) is invariant under the forward iterations of~\eqref{K3_dynamics}, the point \(M_{\txta, 3}^+  \cap \left\{r_3= \rho, \epsilon_3 = \delta, h_3 = \nu \right\} \) lies within \(\Pi_3(\Sigma_3^{\txtin})\). 

Because \(M_{\txta, 3}^+  \cap \left\{r_3= \rho, \epsilon_3 = \delta, h_3 = \nu \right\} \) converges to  \((1,\rho, 0, \nu)\) as \(\delta \to 0\), we infer that for small \(\delta \) the point \(M_{\txta, 3}^+  \cap \left\{r_3= \rho, \epsilon_3 = \delta, h_3 = \nu \right\} \)  and the exponentially small set \(\Pi_3(\Sigma_3^{\textnormal{in}})\)  around this point lie within \(\Sigma_3^{\txtout} = \left\{  x_3 \in 1+ \rho^{-1} J,~r_3= \rho,~\epsilon_3 = \delta,~h_3 = \nu \right\}\). 
\end{proof}

\subsection{Blown-down dynamics}
\label{sec:blowdown}
As a last step, we will connect the individual results, obtained in each of the three charts, to prove Theorem~\ref{maintheorem}.

\begin{proof}[Proof of Theorem~\ref{maintheorem} (T1) and (T2)]
 Throughout the previous proofs we needed to make \(\delta\) sufficiently small; now we choose \(\delta_0 > 0\) such that all of the statements hold true for every \(\delta \in (0,\delta_0]\). This gives the value \(\epsilon_0 := \delta_0 \rho^4\).
 
 As a consequence of condition~\eqref{hrestriction} in the proof of Lemma~\ref{lem:transition3}, we assumed $\nu \leq \frac{1}{3 }$. Even stronger, Lemma \ref{lem:increasing_invariance}.\ref{lem:increasing} gives \(\nu \leq \frac{2}{15}\).
 For Lemma \ref{lem:transition1} we required \(\nu \leq \left({1 + \frac{|J|^2}{8\rho^2}}\right)^{-1}\). 
 Thus we can take \(h_0 := \min \left\{ \frac{2}{15 \rho^2},\frac{1}{\rho^2 + \frac{1}{8}|J|^2} \right\}\).

In the previous sections we have seen that 
\begin{equation}\label{eq:inclusion}
 \Sigma_2^{\txtin}  \subset \kappa_{12}\left(\Sigma_1^{\txtout}\right) \text{ ~~and~~ }  \Sigma_3^{\txtin}  \subset \kappa_{23}\left(\Sigma_2^{\txtout}\right) ~,
\end{equation}

so that we can concatenate the three transition maps \(\Pi_1, \Pi_2\) and \(\Pi_3\)~, using the appropriate coordinate changes in between, and define the map \(\tilde{\Pi}\) from \(\Sigma_1^{\txtin}\) to \(\Sigma_3^{\txtout}\)
\[\tilde{\Pi} := \Pi_3 \circ \kappa_{23} \circ \Pi_2 \circ \kappa_{12} \circ \Pi_1 ~. \]
In particular, we have seen from the analysis in the charts $K_1$, $K_2$ and $K_3$ (see Lemma~\ref{lem:transition1}, Theorem~\ref{mainresult} and Lemma~\ref{lem:transition3}) that $ M_{\txta}^0$, equalling $M_{\txta,1}^0$ in $K_1$ and $ M_{\txta,2}^0$ in $K_2$, is continued to $\Sigma_3^{\txtout}$ in the vicinity of $M_{\txta,3}^+$.
Additionally, we have \(\Delta_{\txtin} = K_1(\Sigma_2^{\txtin})  \) and \(\Delta_{\txtout}^+ = K_3(\Sigma_3^{\txtout})  \). Hence, by reverting to the original coordinates \((x,y,\epsilon, h)\), we obtain 
the transition map $\Pi^+$ from \(\Delta_{\txtin}\) to \(\Delta_{\txtout}^+\) given by 
\[ \Pi^+ := K_3 \circ \tilde{\Pi} \circ K_1^{-1} ~.\]

We have shown in Lemma \ref{lem:transition3} that the set \(\Pi_3 (\Sigma_3^{\txtin}) \subset \Sigma_3^{\txtout}\) has a \(x_3\)-width of 
\(
\left(1-C^*\frac{\nu}{4}\right)^{\frac{3}{4 \nu \delta}} \).
Hence, by transforming back to the blown-down coordinates, it is easy to see that the width in \(x\)- direction of 
 \(K_3(\Pi_3 (\Sigma_3^{\txtin})) \subset \Delta_{\txtout}^+ \)  is \( \bigo{ \left(1-C\cdot h\right)^{\frac{K}{\epsilon h}}} \) when we take \(K= \frac{3\rho^2}{4}\) and \(C = C^*(\theta)\frac{\rho^2}{4}\)  (with \(\theta\) and \(C^*\) from  \eqref{Sigma3in}  and \eqref{def:C*} respectively).

Furthermore, Lemma~\ref{lem:transition1} implies that \(K_1(\Sigma_1^{\txtin})\) contains a point of the slow manifold \(S_{a,\epsilon, h}^0\)~. On the other hand, Lemma \ref{lem:transition3}  shows that the exponentially small set \(K_3(\Pi_3 (\Sigma_3^{\txtin}))\) contains a point of the slow manifold \(S_{a,\epsilon, h}^+\) , and, due to \eqref{eq:inclusion} and the definition of \(\Pi^+\), we also have \[  \Pi^+(\Delta_{\txtin}) \subset K_3(\Pi_3 (\Sigma_3^{\txtin}))  \subset  \Delta_{\txtout}^+ ~. \]
This completes the proof of the statement in Theorem \ref{maintheorem} for the case \(\lambda > 0\). 

Finally, when \(\lambda < 0 \), observe that under the change of variables \(x \mapsto -x\) the \(x\)-equation in system \eqref{discrete} gets transformed into \[-\bar{x} = -x+h \left( -x(y-x^2) + \lambda\epsilon \right), \]
which is equivalent to
\[\bar{x} = x+h \left( x(y-x^2) - \lambda\epsilon \right) ~.
 \] 
Hence, the analysis is the same as for positive \(\lambda\), with the same outcome under symmetric change of variables. 
\end{proof}

\subsection{Canard Case}
\label{canards}

The analysis in the case of \(\lambda = 0\) may be carried out without a blow-up transformation.  Hence, we treat the proof of Theorem~\ref{maintheorem} (T3) separately here.
\begin{proof}[Proof of Theorem~\ref{maintheorem} (T3)]
One observes, that the system

\begin{equation*}
P: 
\begin{pmatrix}
x \\ y
\end{pmatrix} \mapsto 
\begin{pmatrix}
\bar{x} \\ \bar{y}
\end{pmatrix} = 
\begin{pmatrix}
x+h \left( x(y-x^2) \right)  \\
y+h\epsilon
\end{pmatrix}
\end{equation*}

keeps \(\{ x = 0\}\) invariant for all \(\epsilon \geq 0\), since there is no dependency on \(\epsilon\) in the \(x\)-equation.
This means that we have \(S_{a, \epsilon, h}^0 = S_a^0 \) and \(S_{r, \epsilon, h}^0 = S_r^0\) on the domain of our analysis (see discussion around~\eqref{eq:condbigcutoff}) so that \(S_{a, \epsilon, h}^0  \) and \(S_{r, \epsilon, h}^0\) are connected.
The connecting trajectory \(\gamma\) starting in \(\Delta_{\txtin}\) is explicitly given by \(\gamma(n) = \big(x(n), y(n)\big) = \big(0,-\rho^2 + n h\epsilon\big)\).

%The eigenvalue in \(x\)-direction at some point \((0,y_0)\) is given by \(\mu(y_0)= 1+hy_0\).
The linearization along the trajectory $\gamma$ is characterized by the variational equation
\begin{equation}  \label{varequ}
v(n+1) = \begin{pmatrix}
1 +  h y_n  & 0 \\
0 & 1
\end{pmatrix} v(n), \quad v(n) \in \mathbb{R}^2, \fa n \in \mathbb{N}\,.
\end{equation}
While the fixed point $w = (0,1)^{\top}$ of~\eqref{varequ} corresponds with the centre-direction along $\gamma$, the solution of~\eqref{varequ} starting at $v(0) = (1,0)^{\top}$ corresponds with the transversal hyperbolic direction and can be explicitly solved to be 
$$v(n) = (v_1(n), v_2(n))^{\top} = (\prod_{k=0}^{n-1} (1+ h(- \rho^2 + kh\epsilon)), 0)^{\top}.$$
Let us for simplicity assume that we have \(\frac{2\rho^2}{\epsilon h} =:N \in 4\mathbb{N} \). In particular, we set \(\Delta_{\txtout}^0 = \{ (x, \rho^2) | ~ x\in J \}\), where clearly \(\gamma(N) \in \Delta_{\txtout}^0\). 
We have already seen in Section \ref{sec:entering} that due to the cubic structure the contraction rates in \(x\)-direction towards the locally stable fixed point for \(y<0\) are at least as strong as the linear rates achieved by linearization around \(x=0\).  Similarly one also observes that the linear rates give a bound for the expansion rate for positive values of \(y\). 
Hence, the contraction and expansion of trajectories from  $\Delta_{\textnormal{in}}$ to a neighborhood of $\Delta_{\txtout}^0$ can be estimated from above by the linear rate $\mu$ 
along the trajectory \(\gamma\) which satisfies
\begin{align*}
	\mu &= v_1(N) = \prod_{k=0}^{N-1}  1+ h (-\rho^2+k h\epsilon)  \leq \prod_{k=0}^{N}  1+ h (-\rho^2+k h\epsilon)  \nonumber \\ 
	&= \prod_{k=0}^{ N/2} \big( 1+ h (-\rho^2+k h\epsilon)\big)\big( 1+ h (\rho^2-k h\epsilon)\big) = \prod_{k=0}^{ N/2 } \big( 1 - h^2 (\rho^2-k h\epsilon)^2 \big) \nonumber \\ 
	&\leq \prod_{k=0}^{ N/4 } \big( 1 - h^2 (\rho^2-k h\epsilon)^2 \big) \leq \prod_{k=0}^{ N/4 } \big( 1 - h^2 (\tfrac{\rho^2}{2})^2 \big) =\big( 1 -  h^2 (\tfrac{\rho^2}{2})^2 \big) ^{ N/4 }.
\end{align*}
Hence, we can give the bound  
\begin{equation*}
	\mu \leq \Big(1- h^2 (\tfrac{\rho^2}{2})^2\Big)^{\tfrac{\rho^2}{2\epsilon h}} < 1,
\end{equation*}
meaning that the transition map \(\Pi^0\) is contractive for the canard case since the contraction rates along \(S_a^0\) prevail over the expansion rates along \(S_r^0\). Hence, the claim follows.
\end{proof}
However, note that  \[ \lim_{h \to 0} \Big(1- h^2 (\tfrac{\rho}{2})^2\Big)^{\tfrac{2\rho}{\epsilon h}} = 1 .\]

Thus, as expected, in the limit $h \to 0$ one obtains the stability behaviour of the corresponding continuous-time system where contraction and expansion exactly compensate each other. It is still remarkable that the Euler method not only preserves the stability behaviour for trajectories close to the canard but even enhances stability as compared to the continuous-time case for sufficiently small $h > 0$. We already observed this surprising effect in the case of transcritical canards (cf. \cite{EngelKuehn18}) but emphasize that in other similar situations, like the folded canard (cf. \cite{EKPS18}), the Euler method has clearly unfavourable stability properties.
\subsection{Higher Order Terms}\label{sec:hot}

We briefly discuss how our results can be generalized when higher order terms  \(h_1(x,y,\epsilon) = \bigo{x^2y, xy^2, \epsilon x, \epsilon y, \epsilon^2}\) and \( h_2(x,y,\epsilon) = \bigo{x,y,\epsilon}\) from \eqref{eq:normalform} are included.
The corresponding discretized dynamical system reads 
\begin{equation}\label{eq:disc_with_hot}
P: 
\begin{pmatrix}
x \\ y \\ \epsilon \\ h
\end{pmatrix} \mapsto 
\begin{pmatrix}
\bar{x} \\ \bar{y} \\ \bar{\epsilon} \\ \bar{h}
\end{pmatrix} = 
\begin{pmatrix}
x+h \big( x(y-x^2)+\lambda\epsilon + \bigo{x^2y, xy^2, \epsilon x, \epsilon y, \epsilon^2} \big)  \\
y+h\epsilon(1+\bigo{x,y,\epsilon}) \\ \epsilon \\ h
\end{pmatrix}\,.
\end{equation}

Note that due to the dependence of \(\bar{y}\) on \(x\), points in the image of \(\Delta_{\textnormal{in}}\) under iteration of \(P\)  will not share the same \(y\)-coordinate. Thus we cannot define the transition mappings \(\Pi^{\pm, 0}\) by just a fixed number of iterations of \(P\), but instead pointwise for each initial value \((x_0,-\rho^2) \in \Delta_{\textnormal{in}}\) by 
\begin{equation*}
\Pi^{\pm,0} (x_0,-\rho^2) = P^{n^*(x_0)} (x_0, -\rho^2)~, ~~\text{ where } n^*(x_0) = \argmin_{n \in \mathbb{N}} \dist(P^n(x_0, -\rho^2) , \{y=\rho^2\} )~.
\end{equation*}

We see that for a sufficiently small choice of \(\rho\) every trajectory will get \(\epsilon h\) close to \(\{y= \rho^2\}\). In \cite{EngelKuehn18}  this had already to be taken into account for the case of a normal form without higher order terms. 

Firstly, let us discuss the problem for fixed $\lambda \neq 0$. It is an important benefit of the blow-up method that in entering and exiting charts higher order terms have no significant impact. 
In more detail, we transform system \eqref{eq:disc_with_hot} by \(K_1\), proceeding as in Section~\ref{sec:entering}, to obtain
\begin{equation*}
	\bar{r}_1^2 = r_1^2  \left(1-h_1\epsilon_1 + h_1 \epsilon_1 \bigo{r_1x_1\epsilon_1, r_1^2\epsilon_1, r_1^4\epsilon_1^2} \right), 
\end{equation*}
This relation yields
\begin{equation*}
\bar{r}_1\bar{x}_1= r_1x_1 + r_1^{-2}h_1 \left(r_1x_1(-r_1^2-r_1^2x_1^2)+\lambda r_1^4\epsilon_1 + \bigo{x_1^2r_1^4, x_1r_1^5, x_1r_1^5\epsilon_1, r_1^6\epsilon_1} \right).
\end{equation*}
which simplifies to 
\begin{equation*}
\bar{x}_1= \frac{r_1}{\bar{r}_1}\left[ x_1 + h_1 \left(x_1(-1-x_1^2)+\lambda r_1\epsilon_1 + \bigo{x_1^2r_1, x_1r_1^2, x_1r_1^2\epsilon_1, r_1^3\epsilon_1} \right) \right].
\end{equation*}

Consequently the transformed system in $K_1$ can be written as
\begin{align} \label{K1_dynamics_hot}
\begin{array}{rcrcl}
&\bar{x}_1& &=& \big(1-h_1\epsilon_1 + \bigo{r_1} \big)^{-\frac{1}{2}} \big[ x_1+ h_1\big(x_1(-1-x_1^2)+\lambda r_1\epsilon_1 + \bigo{r_1} \big) \big] ~,\\
&\bar{r}_1& &=& \big(1-h_1\epsilon_1 + \bigo{r_1} \big)^{\frac{1}{2}}~ r_1~,\\
&\bar{\epsilon}_1& &=& \big(1-h_1\epsilon_1 + \bigo{r_1} \big)^{-2} \epsilon_1~,\\
&\bar{h}_1& &=&\big(1-h_1\epsilon_1 + \bigo{r_1} \big)~ h_1~.
\end{array}
\end{align}
For \(r_1 = 0\), this system is identical to~\eqref{K1_dynamics}. Hence, for sufficiently small \(r_1\), we still obtain the existence of a center-stable manifold \(M_{a,1}^0\) at the point \(p_a^0(0) \) and the consequences thereof. A small choice of \(r_1\) means that we have to restrict \(\rho\) to sufficiently small values.
For the exiting chart \(K_3\), the situation is similar. 

In the rescaling chart, however, the higher order terms may not be bypassed that easily, but the strategy from Section~\ref{sec:rescaling} can be adapted. 
As in Section \ref{sec:rescaling}, we still have \(\bar{r}_2 = r_2\) and \(\bar{h}_2 = h_2\).
The remaining equations of \eqref{eq:disc_with_hot} transform to
\begin{equation*}
\bar{r}_2\bar{x}_2= r_2x_2 + r_2^{-2}h_2 \left(r_2x_2(r_2^2y_2-r_2^2x_2^2)+\lambda r_2^4 + \bigo{r_2^4x_2^2y_2, r_2^5x_2y_2^2, r_2^5x_2, r_2^6y_2, r_2^8} \right) ,
\end{equation*}
and
\begin{equation*}
\bar{r}_2^2\bar{y}_2 = r_2^2y_2 + r_2^{-2}h_2 r_2^4 \big(1+\bigo{r_2x_2, r_2^2y_2, r_2^4}\big) ~,
\end{equation*}
which can be simplified and desingularized into 
\begin{equation}\label{eq:x2_with_hot}
\bar{x}_2=  x_2 + h_2 \left(x_2(y_2-x_2^2)+\lambda r_2 + \bigo{r_2x_2^2y_2, r_2^2x_2y_2^2, r_2^2x_2, r_2^3y_2, r_2^5} \right) ~, 
\end{equation}
\begin{equation*}
\bar{y}_2 = y_2 + h_2 \big(1+\bigo{r_2x_2, r_2^2y_2, r_2^4}\big) ~.
\end{equation*}

The following arguments will not only require small \(\epsilon\) but also sufficiently small \(\rho\), so that the impact of normal form higher order terms can be controlled and the dynamics are determined by the remaining terms.  
Since the small parameter \(\delta\) incorporates \(\rho\) and \(\epsilon\), it is more apparent in the original not blown up coordinates \((x,y,\epsilon, h)\) how the choice of \(\rho\) determines the considered neighbourhood of the origin. 
Note that in the following small \(\rho\) and \(\epsilon\) mean that the statements hold for sufficiently small fixed \(\rho\) and for all positive \(\epsilon\) below some sufficiently small threshold. 
Restricting \(\rho\) and \(\epsilon\), we can assure upwarded movement in \(y_2\)-direction taking \(\bigo{\nu^{-1}\delta^{-1}}\) steps to travel through the domain considered in the second chart.

Our approach in Section \ref{sec:rescaling} relied heavily on the curve of fixed points \(x_2^*(y_2)\) introduced in Lemma \ref{lem:curve_of_fp}. In the more general setting involving the higher order terms, a curve corresponding to \(x_2^*(y_2)\) persists. In other words, one can show that for fixed, negative values of \(y_2\) there is exactly one positive fixed point for equation \eqref{eq:x2_with_hot}, given that \(\rho\) and \(\epsilon\) (and thus \(\delta\)) are sufficiently small. 
Note that this can only be accomplished for \(y_2\) outside an interval of size \(\bigo{r_2^2}\), since otherwise the term  \(x_2y_2\) is of order   \(\bigo{x_2r_2^2}\) and therefore does not dominate additional terms of that order any longer.
Using the curve of fixed points a result analogous to Proposition~\ref{prop:I1_to_I2} can be shown, which implies that trajectories will enter the quadrant \(\{x>0, y<0\}\). 
Equivalent statements to those of Lemma \ref{lem:increasing_invariance} can also be obtained for \(\rho \) and \(\epsilon\) small enough. This means that for sufficiently small \(\nu\)  the mapping \(x_2 \mapsto \bar{x}_2\) is monotone.
Using the monotonicity one then easily checks that \(\{x_2 \geq 0 \}\) is invariant under \eqref{eq:x2_with_hot}, simply by plugging in \(x_2 = 0\). In a similar manner one can ensure that trajectories leave the rectangular set \(\{x_2 \in [0, \delta^{-\frac{1}{4}}], y_2 \in [-\frac{1}{4}\delta^{-\frac{1}{2}},  \frac{1}{4}\delta^{-\frac{1}{2}} ] \}\) only through its upper boundary. 

Moreover, one can show that in the quadrant \(\{ x \geq 0, y \geq 0\}\) trajectories are bounded away from the \(y_2\)-axis, independently from \(\epsilon\). The corresponding result is found in Proposition~\ref{prop:I3_to_I4}.
Hence, we may deduce that all trajectories will end up at a \(y_2\)-height close to \(\frac{1}{4}\delta^{-\frac{1}{2}}\) with positive \(x_2\)-value, bounded away from 0 and smaller than \(\delta^{-\frac{1}{4}}\). In other words, we obtain a result similar to Theorem \ref{mainresult}.
Transforming the exiting set into \(K_3\)-coordinates allows to proceed similarly as in Section~\ref{sec:exiting}. As we have already seen in chart \(K_1\), the higher order terms do not change the behaviour for sufficiently small \(\rho\). Summarizing, we deduce that Theorem \ref{maintheorem}, (T1)  and (T2), can be transferred to the general setting of \eqref{eq:disc_with_hot}.

Furthermore, note that in the general case of the normal form~\eqref{eq:normalform} including higher order terms, the value of $\lambda$ close to $0$ giving a canard changes with the value of  $\epsilon$. For continuous time, this phenomenon is studied in detail for canards in folds in \cite{ks2011} and discussed for the transcritical and pitchfork case in \cite[Remark 2.2 and Remark 4.1]{ks2001/2}.
Using a Melnikov computation, one may show the
existence of a function $\lambda_c(\epsilon^{1/4})$ with $\lambda_c(0) = 0$ such that for $\lambda =  \lambda_c(\epsilon^{1/4})$ the slow manifold
$S_{\txta, \epsilon}^- $ extends to $S_{\txtr, \epsilon}^+$  for sufficiently small $\epsilon$. In Theorem~\ref{maintheorem} (T3), we only treated the case $\lambda_c(\epsilon^{1/4}) \equiv 0$ since we did not take into account perturbations from higher order terms. In order to obtain an analogous result to the ODE case, a treatment of the more general problem~\eqref{eq:disc_with_hot} about $\lambda =0$ requires a discrete Melnikov computation, which is more complicated. Therefore, we are going to treat the general canard problem in the separate study \cite{EKPS18}.

%\section{Summary}
%\label{sec:outlook}
%
%
%The paper has treated an Euler discretization of the normal form of a fast slow system exhibiting a pitchfork singularity. We have utilized the blow-up method to analyse the behaviour of trajectories near the singularity at the origin.  
%We have shown that, for $\lambda \neq $, the continuation of the single attracting branch of the slow manifold to one of the parabolic branches, depending on the sign of \(\lambda\), is preserved under discretization and have computed a rate of contraction to the manifold. The special case of \(\lambda = 0\) is characterized by a Canard phenomenon which is treated here for the simplified normal form without perturbations by higher order terms, similarly to \cite{EngelKuehn18} (see \cite{EKPS18} for a more detailed analysis). In our analysis, we observe this Canard problem also as a limiting phenomenon, since our estimate shows that contraction to the attracting branch away from the repelling manifold is lost as \(\lambda \to 0 \).
%
%Forming the limit \(h \to 0\) in Theorem \ref{maintheorem}, one obtains the analogous result for the corresponding continous-time system.
%This problem has already been treated by Krupa and Szmolyan \cite{ks2001/2}, using a Melnikov-type integration supplementing the blow-up method. In contrast, our proof has relied only on direct trajectory based estimates. 

\newpage

\bibliographystyle{plain}
\bibliography{mybibfile}

\end{document}